\UseRawInputEncoding
\documentclass[12pt]{amsart}
\usepackage{epsfig, color, amsmath, esint, hyperref, mathrsfs, xcolor, bm, enumitem, mathtools, comment, amsfonts, amssymb}

\usepackage[backend=biber,style=alphabetic]{biblatex}
\hypersetup{colorlinks}
\hypersetup{citecolor=blue}
\hypersetup{urlcolor=blue}
\makeatother
\theoremstyle{definition}
\def\fnum{equation} 
\newtheorem{Thm}[\fnum]{Theorem}
\newtheorem{Cor}[\fnum]{Corollary}

\newtheorem{Lem}[\fnum]{Lemma}

\newtheorem{Exa}[\fnum]{Example}
\newtheorem{Rem}[\fnum]{Remark}
\newtheorem{Pro}[\fnum]{Proposition}
\newtheorem{Obs}[\fnum]{Observation}
\newtheorem{Not}[\fnum]{Notation}
\newtheorem{Def}[\fnum]{Definition}
\newtheorem{mainthm}{Theorem}

\numberwithin{equation}{section}
\newcommand{\diam}{{\text {diam}}}
\newcommand{\supp}{{\text {supp}}}
\newcommand{\inte}{{\text {int}}}
\newcommand{\dime}{{\text {dim}}}
\newcommand{\conv}{{\text {conv}}}
\newcommand{\Id}{\text{Id}}
\newcommand{\iso}{\text{Iso}}
\newcommand{\vv}{\mathrm{v}}
\newcommand{\ww}{\mathrm{w}}
\newcommand{\vol}{\text{vol}}
\newcommand{\rcd}{\text{RCD}}
\newcommand{\ric}{\text{Ric}}

\newcommand{\mm}{\mathfrak{m}}

\DeclareMathOperator{\Geo}{Geo}

\title{Non-collapsed eGH convergence and dimension}
\author{Jes\'us N\'u\~nez-Zimbr\'on}
\address{\parbox{\linewidth}{Jes\'us N\'u\~nez-Zimbr\'on\\ Universidad Nacional Aut\'onoma de M\'exico\\ nunez-zimbron@ciencias.unam.mx}}
\author{Jaime Santos-Rodr\'iguez} \address{\parbox{\linewidth}{Jaime Santos-Rodr\'iguez\\ Universidad Polit\'ecnica de Madrid\\ jaime.santos@upm.es}} 
\author{Sergio Zamora}
\address{\parbox{\linewidth}{Sergio Zamora\\ Oregon State University\\  zamorabs@oregonstate.edu}}

\begin{document}

\maketitle

\begin{abstract}
Let  $(X_i,p_i)$  be a non-collapsing sequence of pointed $n$-dimensional Riemannian manifolds with a uniform lower Ricci curvature bound, and  $G_i \leq \iso (X_i)$ a sequence of closed subgroups of isometries. 
We show that if the triples $(X_i, G_i, p_i)$ converge in the equivariant Gromov--Hausdorff sense to a triple $(X,G,p)$, then  $\dime (G) \geq \limsup _{i \to \infty} \dime (G_i)$, 
generalizing a result of Mazur--Rong--Wang to the non-compact setting.

The argument also applies in the non-smooth setting of $\rcd$ spaces.
As an application, we investigate $\rcd$ spaces with large isometry groups, 
extending results of Galaz-Garc\'ia--Kell--Mondino--Sosa and Galaz-Garc\'ia--Guijarro.

\end{abstract}

{
\hypersetup{linkcolor=black}
\tableofcontents
}

\section{Introduction}

Equivariant Gromov--Hausdorff convergence has proven itself to be a 
successful tool for studying the geometric and topological structure of spaces 
with lower curvature bounds (see for example \cite{fukaya-yamaguchi, kapovitch-petrunin-tuschmann, kapovitch-wilking, rong}).
The goal of this paper is to investigate certain relationships 
between the sequence of groups and the limit group in this context.

Let  $(X_i,p_i)$  be a sequence of proper pointed metric spaces and 
$G_i \leq \iso (X_i)$ a sequence of closed groups of isometries.  
Assume 
\[  (X_i,G_i, p_i) \xrightarrow{eGH} (X,G,p) .   \]
Let us recall some established results in this context:
\begin{enumerate}
    \item In the case that the spaces $X_i$ are Riemannian manifolds with a uniform lower bound on sectional curvature, it was proved by Fukaya-Yamaguchi \cite{fukaya-yamaguchi-1994} that $\iso(X)$ (and thus also $G$) is a Lie group which is compact if $X$ is compact. The same result in the case of a uniform lower bound on the Ricci curvature was obtained by Colding--Naber \cite{colding-naber}. These results were extended to the context of Alexandrov spaces by Galaz-Garc\'ia--Guijarro \cite{galaz-guijarro} and to the setting of $\rcd$ spaces by Guijarro--Santos-Rodríguez \cite{guijarro-santos} and Sosa \cite{sosa}.
    \item By the work of Grove--Karcher--Ruh \cite{grove-karcher-ruh}, if the orbits $G_i \cdot p_i$ are uniformly bounded and $G$ is a Lie group, then the maps $\phi _i : G_i \to G$ demonstrating the convergence 
can be taken to be continuous group homomorphisms for $i$ large enough (see also \cite{harvey, alattar}). 
    \item If in addition $(X_i,p_i)$ is a non-collapsing sequence of $n$-dimensional Riemannian manifolds 
(or more generally, Alexandrov spaces) 
with a uniform lower sectional curvature bound,  it was proven independently by Mazur--Rong--Wang \cite{mazur-rong-wang} and Harvey \cite{harvey} that  the continuous group homomorphisms $\phi_i$ are injective for $i$ large enough. In particular, this implies 
\begin{equation}\label{eq:dim-semicont}
      \dime  (G) \geq \limsup _{i \to \infty } \dime (G_i) . 
\end{equation}
\end{enumerate}

Of particular interest to us are two natural questions raised by Harvey:
\begin{itemize}
    \item Can the Alexandrov hypothesis be weakened to a lower Ricci curvature bound?
    \item Does \eqref{eq:dim-semicont} hold without the uniform boundedness
    assumption on the orbits?
\end{itemize}
In the proof of %\cite[Proposition 4.1]{harvey}
\cite{mazur-rong-wang}, the curvature condition was only used to show that the sequence $\iso (X_i)$ has the \emph{no small subgroup property} as in Definition \ref{def:ss1}. Since  non-collapsing sequences of Riemannian manifolds with a uniform lower Ricci curvature bound also have the no small subgroup property \cite[Theorem 0.8]{pan-rong}, the proof in %\cite[Proposition 4.1]{harvey} 
\cite{mazur-rong-wang} carries over to that setting, giving a positive answer to the first question (see also \cite[Theorem C]{alattar} for a similar result in a more general setting). 

On the other hand, the second question is much more difficult to address since when $G$ is not compact, the maps $\phi_i : G_i \to G $ demonstrating the convergence cannot in general be taken to be group homomorphisms.   For example, for any Lie group $G$ and any Riemannian metric $g_0 $ on $G$, if $X_i : = (G, i g_0)$ for each $i \in \mathbb{N}$, then $(X_i , G , e) \xrightarrow{eGH} (\mathbb{R}^n , \mathbb{R}^n , 0)$ with $n = \dim (G)$. However, many Lie groups $G$ do not admit non-trivial group homomorphisms $G \to \mathbb{R}^n$.

Nevertheless,  a version of \eqref{eq:dim-semicont} in this setting  was obtained in 
\cite[Corollary 3.3]{pan-ye} assuming the groups are torsion-free abelian.
The first main result of this paper establishes \eqref{eq:dim-semicont}
without any additional assumption on the groups, answering Harvey's second question in the affirmative.  

\color{black}

\begin{mainthm}\label{thm:main}
    Let $(X_i,p_i)$ be a sequence of pointed complete $n$-dimensional Riemannian manifolds with 
    \begin{equation}\label{eq:non-collapsed}
      \ric (X_i) \geq - (n-1), \qquad \vol (B_1(p_i)) \geq v > 0 ,      
    \end{equation}
     and let $G_i \leq \iso (X_i)$ be a sequence of closed subgroups. If 
    \[    ( X_i , G_i, p _i) \xrightarrow{eGH} (X,G, p)    ,           \]
    then 
    \[   \dime (G) \geq \limsup _{i \to \infty } \dime (G_i).          \]
\end{mainthm}

Arguing as in \cite[Corollary 4.2]{harvey}, we obtain the following corollary,
which shows that the symmetry degree is upper semi-continuous under non-collapsed convergence.

\begin{Cor}\label{cor:dimension-semicontinuity}
    Let $(X_i,p_i)$ be a sequence of pointed complete $n$-dimensional Riemannian manifolds 
    satisfying \eqref{eq:non-collapsed}. If 
    \[    ( X_i , p _i) \xrightarrow{pGH} (X, p)    ,           \]
    then 
    \[   \dime (\iso (X)) \geq \limsup _{i \to \infty } \dime (\iso ( X_i)).          \]
\end{Cor}

\subsection{Small subgroups}
Let $(X,p)$ be a pointed proper metric space. We equip  $\iso (X)$ with the metric
\[   d_{p} (g,h) : = \inf _{r >0} \left\{ \frac{1}{r} +  \sup _{x \in B_r(p) }   d(gx,hx)
 \right\} .                         \]
 This metric induces the compact-open topology, is left-invariant, and turns $\iso (X)$ into a proper metric space. 
\begin{Def}\label{def:ss1}
Let $(X_i,p_i)$ be a sequence of pointed proper metric spaces. We say a sequence of groups $H_i \leq \iso (X_i)$ is \emph{small} or consists of  \emph{small subgroups} if for all $r > 0 $ one has
\[    \lim_{i \to \infty} \sup _{h \in H_i} \sup_{x \in B_r(p_i)} d(hx,x) = 0 ,    \]
or equivalently, 
\[   \lim_{i \to \infty} \diam _{d_{p_i}} H_i = 0 . \]  
Given a sequence of groups $G_i \leq \iso (X_i)$, we say it has the \emph{no small subgroup} (NSS) property, or \emph{doesn't  admit small subgroups} if any sequence of small subgroups $H_i \leq G_i$ is eventually trivial.
\end{Def}

If a sequence of pointed Riemannian manifolds $(X_i,p_i)$ satisfies \eqref{eq:non-collapsed}, then the sequence $\iso (X_i)$ has the NSS property, so the following is a more general version of Theorem \ref{thm:main}.

\begin{mainthm}\label{thm:largest-small-metric}
Let $(X_i,p_i)$ be a sequence of pointed proper metric spaces, and $G_i \leq \iso (X_i)$ a sequence of closed subgroups. Assume
\[     (X_i, G_i, p_i) \xrightarrow{eGH} (X,G,p)                         \]
with $G$ a Lie group of dimension $k$. Then there is a sequence of small subgroups $H_i \leq  G_i$ with the property that any sequence of small subgroups $H_i ' \leq G_i$ satisfies $H_i ' \leq H_i$ for $i$ large enough. Moreover, if there is $R>0$ such that 
\begin{equation}\label{eq:boundedly-generated}
     G_i = \langle \{ g \in G_i \vert d(gp_i,p_i) \leq R \}  \rangle     
\end{equation}
for $i$ large enough, then $H_i $ is normal in $G_i$ and:
\begin{itemize}
    \item The sequence $G_i / H_i \leq \iso (X_i / H_i) $ has the NSS property.
    \item For $i$ large enough, $G_i / H_i $ is a Lie group of dimension $\leq k$.
\end{itemize}
\end{mainthm}

The first part of Theorem \ref{thm:largest-small-metric} is mostly a metric reformulation of \cite[Proposition 9.2]{breuillard-green-tao}, so the main new contribution is the bound on the dimensions of the groups $G_i / H_i$.

\begin{Rem}\label{rem:rcd}
  The hypothesis in Theorem \ref{thm:main} that the spaces $X_i$ are Riemannian manifolds satisfying \eqref{eq:non-collapsed} is only used to guarantee that $G$ is a Lie group and that the sequence $G_i$ has the NSS property. 
  Hence said hypothesis can be replaced with the weaker hypothesis that the metric spaces $X_i$ and $X$ admit measures making them $\rcd (K,N)$ spaces of the same essential dimension for some $K \in \mathbb{R}$, $N \geq 1$ 
   (see \cite{guijarro-santos, sosa} and \cite[Theorem 93]{santos-zamora}).  
\end{Rem}

\subsection{Good approximations} In most of this paper, we study the almost homomorphisms provided by the definition of equivariant Gromov--Hausdorff convergence 
without explicitly making reference to the metric spaces they act on. The ensuing definition is inspired by the finite approximations introduced in \cite{turing} 
and the good models introduced in \cite{hrushovski} (see also \cite[Definition 57]{santos-zamora}).

\begin{Not}
For a subset $A$ of a  group  $G$, we denote by $A^n$ the set of elements $a_1 \cdots a_n \in G$ with $a_j \in A$ for each $j$. We say that $A$ is \emph{symmetric} if it contains the identity and is closed under inverses.   
\end{Not}

\begin{Def}\label{def:ga}
    Let $G_i$ be a sequence of locally compact Hausdorff groups and $G$ a locally compact Hausdorff group. We say a sequence of functions $\phi _ i : G_i \to G$ consists of \emph{good approximations} if there are symmetric, open, pre-compact sets $A_i \subset G_i$, $A \subset G$ such that:
    \begin{enumerate}[label=\Roman*]
        \item (Almost surjectivity) For all $U \subset A $ open nonempty, one has $\phi_i (A_i) \cap U \neq \emptyset$ for $i$ large enough.\label{item:ga-1}
        \item (No expansion) For all $V \subset G$ open with $\overline{A} \subset V$, one has $\phi_i (A_i) \subset V$ for $i$ large enough.\label{item:ga-2}
        \item (Almost homomorphism) For each $n \in \mathbb{N}$, and sequences  $g_i, h_i \in A_i ^n $, one has
        \[  \lim_{i \to \infty}  \phi_i (g_ih_i)^{-1} \phi _i (g_i) \phi_i (h_i) = e_G   . \]\label{item:ga-3}
        \item (No compression) For each $n \in \mathbb{N}$ and each compact $K\subset A$, %there is $i_0 \in \mathbb{N}$  such that if $i \geq i_0$, and $g \in A_i^n$ satisfies  $\phi_i (g) \in K $,  then $g \in A_i$.\label{item:ga-4} 
        one has 
        \[  \phi_i ^{-1} (K) \cap A_i ^n \subset A_i        \]
        for $i$ large enough. \label{item:ga-4}
        \item (Almost continuity) For all $U \subset G$ identity neighborhood there is a sequence $U_i \subset G_i$ of identity neighborhoods such that $\phi_i (U_i) \subset U$ for $i$ large enough. \label{item:ga-5}
    \end{enumerate}
    If properties (I--V) hold, we call the sets $A_i$ and $ A$ \emph{regular neighborhoods} with respect to the approximations $\phi_i$. 
\end{Def}

\begin{Def}\label{def:ss2}
    Let $\phi_i : G_i \to G$ be a sequence of good approximations with regular neighborhoods $A_i \subset G_i$. We say a sequence of subgroups $H_i \leq G_i$ is \emph{small} or \emph{consists of small subgroups} with respect to the pairs $(\phi_i, A_i)$  if $H_i \subset A_i$ for $i$ large enough, and 
    \[  \phi_i (h_i) \to e \text{ for any sequence }h_i \in H_i . \]
    We say the sequence $G_i$  has the \emph{no small subgroup} (NSS) property or \emph{doesn't admit small subgroups} with respect to the pairs $(\phi_i, A_i)$ if any sequence of small subgroups $H_i \leq G_i$ with respect to the pairs $(\phi_i, A_i)$ is eventually trivial.
\end{Def}

\begin{Rem}
When clear from the context, we will omit the dependence on the pairs $(\phi_i, A_i)$ when talking about sequences of small subgroups.    
\end{Rem}

The following theorem confirms that equivariant Gromov--Hausdorff convergence produces good approximations, and that the different notions of small subgroups agree.

\begin{mainthm}\label{thm:egh-to-ga}
    Let $(X_i, p_i)$, $(X,p)$ be pointed proper metric spaces and $G_i \leq \iso (X_i)$, $G \leq \iso (X)$ closed groups of isometries such that 
    \[     (X_i, G_i, p_i) \xrightarrow{eGH} (X,G,p) .  \]
    Then the maps $\phi_i : G_i \to G$ given by the definition of equivariant Gromov--Hausdorff convergence are good approximations.  Moreover, a sequence of subgroups $H_i \leq G_i$ is small as in Definition \ref{def:ss1} if and only if it is small as in Definition \ref{def:ss2}. 
\end{mainthm}

Theorem \ref{thm:egh-to-ga} allows one to reduce Theorems \ref{thm:main} and \ref{thm:largest-small-metric} to Theorems \ref{thm:model} and \ref{thm:largest-small} below, respectively.

\begin{Thm}\label{thm:model}
    Let $G_i$ be a sequence of Lie groups and  $\phi _i : G_i \to G$ a sequence of good approximations. If $G$ is a Lie group and  the groups $G_i$ have the NSS property, then
    \[    \text{dim} (G) \geq \limsup_{i \to \infty} \text{dim} (G_i) .    \] 
\end{Thm}

\begin{Thm}\label{thm:largest-small}
    Let  $\phi _i : G_i \to G$ be a sequence of good approximations with regular neighborhoods $A_i \subset G_i$ and $A \subset G$. If $G$ is a Lie group, then there is a sequence of subgroups $H_i  \leq G_i $ such that:
    \begin{itemize}
        \item The sequence $H_i$ is small.
        \item $H_i \trianglelefteq \langle A_i \rangle $ for $i$ large enough.
        \item For any sequence $H_i ' \leq G_i$ of small subgroups, one has $H_i ' \leq H_i$ for $i$ large enough. 
        \item $G_i ' : = \langle A_i \rangle / H_i$ is a Lie group for $i$ large enough.
    \end{itemize}
\end{Thm}

\subsection{Outline}
Theorem \ref{thm:model} is the main technical result of this paper. 
The first step in the proof is to construct 
a sequence of good approximations $\psi _i : G_i \to \mathbb{R}^k$ 
with the NSS property, where $k : = \dime (G)$. 
This is achieved by Lemma \ref{lem:zoom-ii}. 
Arguing by contradiction, we can assume 
$n_i : = \dime (G_i) > k$ for all $i$. 
Given a norm $\Vert \cdot \Vert _i$ on $\mathfrak{g}_i$, the Lie algebra of $G_i$, 
we can identify $\mathbb{S}^{n_i-1}$ with the unit vectors of $\mathfrak{g}_i$.
If the norm is chosen carefully, the map 
$f_i : \mathbb{S}^{n_i -1} \to \mathbb{R}^k$ 
given by $f_i (\vv) : = \psi _i (\exp (\vv))$ is approximately continuous, 
almost respects antipodals, and its image stays away from $0$. 
This violates an approximate version of the Borsuk--Ulam Theorem (Theorem \ref{thm:abu}).
The main difficulty is the construction of the norm $\Vert \cdot \Vert _i$, 
which relies on the Gleason Lemmas of 
Breuillard--Green--Tao (see Theorem \ref{thm:gleason}).

This paper is organized as follows. 
In Section \ref{sec:rcd} we present applications of our main results to the theory of RCD spaces. 
In Section \ref{sec:examples} we provide examples showing why the hypotheses of the main theorems are necessary, 
as well as examples illustrating what happens when each condition in the definition of good approximation is removed. 
In Section \ref{sec:prelims} we review the background material used in this paper. 
In Section \ref{sec:egh-to-ga} we prove Theorem \ref{thm:egh-to-ga}. 
In Section \ref{sec:blowup} we construct good approximations with target $\mathbb{R}^k$ from good approximations with target a Lie group. 
In Section \ref{sec:max-small} we prove Theorem \ref{thm:largest-small}. 
In Section \ref{sec:dimension} we prove Theorems \ref{thm:main} and \ref{thm:largest-small-metric}. 
Finally, in Section \ref{sec:rcd-proofs} we give the proofs of the results stated in Section \ref{sec:rcd}.

\section{RCD spaces and applications}\label{sec:rcd}

For $K \in  \mathbb{R}$, $N \in  [1, \infty) $, the class of $\rcd (K,N)$
spaces consists of proper metric measure spaces that satisfy a synthetic condition 
of having Ricci curvature bounded below by $K$, dimension bounded above by $N$, and being infinitesimally Hilbertian.
$\rcd (K,N)$ spaces were first introduced by Ambrosio--Gigli--Savar\'e in \cite{ambrosio-gigli-savare, gigli} 
building upon the curvature-dimension condition developed by Sturm and Lott--Villani \cite{lott-villani, sturm}, and have been investigated 
extensively in the past decade.

We note that a large number of papers in the literature work with a condition known as $\rcd  ^{\ast}(K, N)$, originally introduced in \cite{bacher-sturm}. It is now known, thanks to the work of Cavalletti--Milman \cite{cavalletti-milman} and Li \cite{li}, that this condition is equivalent to the $\rcd (K, N)$ condition.

A significant milestone, due to Bru\'e--Semola, is that any $\rcd (K,N)$ space $(X,d , \mm )$ admits a 
unique integer $n \in [0, N]$ such that for $\mm$-almost every $x \in X$, one has
\[    (\lambda _i X, x )  \xrightarrow{pGH} (\mathbb{R}^n, 0)   \]
for any sequence $\lambda _i \to \infty $ \cite{brue-semola}. Such points are called $n$-\emph{regular points}, and the integer $n$ is called
the \emph{essential dimension} of $X$.

It is well known that for an $n$-dimensional Riemannian manifold $X$, if 
$G \leq \iso (X)$, then 
\begin{equation}\label{eq:max-symmetry}
    \dime (G) \leq \frac{n(n+1)}{2} ,  
\end{equation}
with equality only if after re-scaling, $X$ is isometric to one of  
$\mathbb{S}^n,$ $\mathbb{R}P^n$, $\mathbb{R}^n$, or $\mathbb{H}^n$.

For an $\rcd (K,N)$ space $(X,d, \mm)$, it was established in 
\cite{guijarro-santos, sosa} that $\iso (X)$ is a Lie group. Further, in \cite{guijarro-santos} it was shown that if $G \leq \iso (X)$, then 
\[   \dime (G) \leq \frac{\lfloor N \rfloor \lfloor N+1 \rfloor }{2} ,   \]
with equality only if after re-scaling,  $X$ is isometric to one of  
$\mathbb{S}^{ \lfloor N \rfloor } ,$ $\mathbb{R}P^{ \lfloor N \rfloor }$, $\mathbb{R}^{ \lfloor N \rfloor }$, or $\mathbb{H}^{ \lfloor N \rfloor }$.

In \cite[Theorem 6.7]{galaz-kell-mondino-sosa}, 
Galaz-Garc\'ia--Kell--Mondino--Sosa obtained an analogous result 
using the essential dimension in place of $ \lfloor N \rfloor $, 
assuming that $G$ is compact and that the $G$-action 
is Lipschitz, co-Lipschitz, and measure-preserving.  
As a consequence of our main results, 
we can remove these additional hypotheses.

\begin{Thm}\label{thm:max-symmetry}
Let $(X,d,\mm)$ be an $\rcd (K,N)$ space of essential dimension $n$, 
and $G\leq \iso (X)$ a closed subgroup. Then \eqref{eq:max-symmetry} holds, 
with equality only if after re-scaling,  $X$ is isometric to one of  
$\mathbb{S}^n$, $\mathbb{R}P^n$, $\mathbb{R}^n$, or $\mathbb{H}^n$.
\end{Thm}

We also obtain a similar bound for groups not acting transitively, 
extending the first part of \cite[Theorem 6.8]{galaz-kell-mondino-sosa}.
Again, we remove the hypotheses that $G$ is compact and that  
its action is Lipschitz, co-Lipschitz, and measure-preserving.

\begin{Thm}\label{thm:not-transitive}
Let $(X,d,\mm)$ be an $\rcd (K,N)$ space of essential dimension $n$, 
and $G\leq \iso (X)$ a closed subgroup not acting transitively. 
Then 
\begin{equation}\label{eq:nt-max-symmetry}
\dime (G) \leq \frac{n(n-1)}{2} ,
\end{equation}
with equality only if $X/G$ is isometric to a real interval or a circle.
\end{Thm}

Note that the hypothesis that $G$ does not act transitively 
is automatically satisfied if $X$ has non-empty boundary, 
as defined in \cite{kapovitch-mondino}. Therefore, 
Theorem \ref{thm:not-transitive} extends \cite[Theorem 7.1]{galaz-guijarro}
from Alexandrov spaces to $\rcd$ spaces\footnote{The bound stated in \cite[Theorem 7.1]{galaz-guijarro} should read $\dime(\iso (X))\leq n(n-1)/2$, instead of $\dime (\iso (X)) \leq (n-1)(n-2)/2 $.}.

Let $(X,d,\mm)$ be an $\rcd (K,N)$ space and let $G \leq \iso (X)$ be a compact subgroup. By \cite{santos-r},
the measure $\tilde{\mm}$ obtained by averaging the $G$-orbit of $\mm$ is $G$-invariant and makes   $(X,d,\tilde{\mm})$ an 
$\rcd (K,N)$ space. Then,  by \cite{galaz-kell-mondino-sosa}, 
$ ( X^{\ast} , d ^{\ast}, \mm ^{\ast})$  is also an $\rcd (K,N)$
space, where $X^{\ast} = X/G$, $d^{\ast}$ is the quotient metric, 
and $\mm ^{\ast}$ the pushforward of $\tilde{\mm}$  
under the projection  $X \to X^{\ast}$. In this setting, we can also bound the dimension of $G$ in terms of the difference between the essential dimensions of $X$ and $X/G$.

\begin{Thm}\label{thm:dimension-difference}
Let $(X,d,\mm)$ be an $\rcd (K,N)$ space of essential dimension $n$, 
and let $G\leq \iso (X)$ be a compact subgroup. If the $\rcd (K,N)$ space
$(X^{\ast}, d^{\ast}, \mm ^{\ast})$ described above has 
essential dimension $m$, then 
\begin{equation}\label{eq:dimension-difference}
\dime (G) \leq \frac{(n-m)(n-m+1)}{2} .
\end{equation}
\end{Thm}

\section{Examples}\label{sec:examples}

The following example shows that Theorem \ref{thm:main} fails if one removes the non-collapsing condition. 

\begin{Exa}
    Let $X_i$ be the $n$-dimensional round sphere of radius $1/i$, and $G_i : = \iso (X_i)$. Then $\iso (X_i)$ has dimension $n(n+1)/2$. However, for any $p_i \in X_i$ one has 
    \[   (X_i, G_i , p_i) \xrightarrow{eGH} (\{ \ast \} , \{ e \} , \ast ) .   \]
\end{Exa}

The following example shows that under the conditions of Theorem \ref{thm:largest-small-metric}, if \eqref{eq:boundedly-generated} fails, then the maximal small subgroups $H_i$ are in general not normal in $G_i$. 

\begin{Exa}\label{ex:max-small-no-normal}
    Let $V : = \{ a,b,c, x, y, z \} $ and $E: = \{ (a,x), (b,y), (c,z), (a,b), (b,c), (c,a) \} $. Let $X_i$ be the graph $(V,E)$ equipped with the length metric in which each edge has length $i$. Let $p_i \in X_i $ be the vertex corresponding to $x$, and $G_i : = \iso (X_i) $. Notice that $G_i$ can be identified with the group of permutations of $\{ a,b,c\}$.  It is then not hard to show that 
    \[       (X_i,G_i ,p_i) \xrightarrow{eGH} ([0, \infty), \{ e \} , 0 )  .       \]
    The sequence of  groups $H_i : = \{ e, (b\,c ) \} \leq G_i $ is small since both elements of $H_i$  fix every point at distance $\leq i$ from $p_i$.  The sequence $H_i$ is maximal small, since any element of $G_i \backslash H_i$ moves $p_i$ to another leaf of the graph, which is at distance $3i$. However, it is clear that $H_i$ is not normal in $G_i$. 
\end{Exa}

The following example shows that if one removes from Theorem \ref{thm:largest-small} the hypothesis that  $G$ is a Lie group, then maximal small subgroups $H_i$ may fail to exist.

\begin{Exa}\label{ex:no-max-small}
    Let $G$ be a first countable locally compact Hausdorff group that is not a Lie group (for example, the $p$-adic integers).  Let $G_i = G$ for all $i \in \mathbb{N}$ and let $\phi_i : G _i \to G$ be the identity for all $i \in \mathbb{N}$. A sequence of small subgroups $H_i \leq G_i$ consists of a sequence of subgroups of $G$ with $h_i \to e$ for all $h_i \in H_i$. 

    Let $H_i \leq G_i$ be any sequence of non-trivial small subgroups. By simply reindexing and repeating elements of the sequence, it is not hard to construct another sequence of small subgroups $H_i ' \leq G_i$ with $H_i' \nleq H_i $ for all $i$ large enough. This implies that there is no sequence of maximal small subgroups.

%    For any sequence of small subgroups $H_i \leq G_i$, it is not hard to construct another sequence of small subgroups $H_i ' \leq G_i$ with $H_i' \nleq H_i $ for all $i$ large enough. This implies that there is no sequence of maximal small subgroups.  
\end{Exa}

\begin{comment}
    
The following examples illustrate the roles of conditions (I--V) in the definition of good approximation.

\begin{Exa}
Let $G_i = \mathbb{Z}$, $G = \mathbb{R}$, $A_i : = \{ -1, 0, 1 \}$, $A : = (- 1 , 1 )$, and $\phi_i : G_i \to G$ the natural inclusion.  The maps $\phi _i $ satisfy all the conditions of Definition \ref{def:ga} except \eqref{item:ga-1}.

Let $G_i = G = \mathbb{Z}$, $A_i : = \{ -i, \ldots , i \}$, $A: = \{ -1, 0 , 1 \} $, and $\phi_i : G_i \to G$ the identity.  The maps $\phi _i $ satisfy all the conditions of Definition \ref{def:ga} except \eqref{item:ga-2}.

Let $G$ and $H$ be two non-isomorphic groups for which there is a homeomorphism $\phi : H \to G$ that respects inverses (for example, two simply-connected nilpotent Lie groups of the same dimension). Let  $A \subset G$ be an open symmetric pre-compact set,  $G_i : = H$,  $\phi_i : = \phi $, and $A_i : = \phi_i^{-1}(A)$.  Then the maps $\phi _i $ satisfy all the conditions of Definition \ref{def:ga} except \eqref{item:ga-3}.

Let $G_i : = \mathbb{R}^2$, $G : = \mathbb{R}$, $A_i : = B_1 (0) \subset G_i$, $A : = (-1,1) $, $\phi_i : G_i \to G$ as $\phi_i (x,y) = x$. The maps $\phi _i $ satisfy all the conditions of Definition \ref{def:ga} except \eqref{item:ga-4}.

Let $G_i = G = A_i = A = \mathbb{S}^1$, and let $\phi_i : G_i \to G$ be a non-continuous group homomorphism with dense image.  The maps $\phi _i $ satisfy all the conditions of Definition \ref{def:ga} except \eqref{item:ga-5}.
\end{Exa}

\end{comment}

\begin{Exa} Here we illustrate the roles of conditions (I--V) in the definition of good approximation.
\begin{itemize}
    \item Let $G_i = \mathbb{Z}$, $G = \mathbb{R}$, $A_i : = \{ -1, 0, 1 \}$, $A : = (- 1 , 1 )$, and $\phi_i : G_i \to G$ the natural inclusion.  The maps $\phi _i $ satisfy all the conditions of Definition \ref{def:ga} except \eqref{item:ga-1}. 
    \item Let $G_i = G = \mathbb{Z}$, $A_i : = \{ -i, \ldots , i \}$, $A: = \{ -1, 0 , 1 \} $, and $\phi_i : G_i \to G$ the identity.  The maps $\phi _i $ satisfy all the conditions of Definition \ref{def:ga} except \eqref{item:ga-2}.
    \item Let $G$ and $H$ be two non-isomorphic groups for which there is a homeomorphism $\phi : H \to G$ that respects inverses (for example, two simply-connected nilpotent Lie groups of the same dimension). Let  $A \subset G$ be an open symmetric pre-compact set,  $G_i : = H$,  $\phi_i : = \phi $, and $A_i : = \phi_i^{-1}(A)$.  Then the maps $\phi _i $ satisfy all the conditions of Definition \ref{def:ga} except \eqref{item:ga-3}.
    \item Let $G_i : = \mathbb{R}^2$, $G : = \mathbb{R}$, $A_i : = B_1 (0) \subset G_i$, $A : = (-1,1) $, $\phi_i : G_i \to G$ as $\phi_i (x,y) = x$. The maps $\phi _i $ satisfy all the conditions of Definition \ref{def:ga} except \eqref{item:ga-4}.
    \item Let $G_i = G = A_i = A = \mathbb{S}^1$, and let $\phi_i : G_i \to G$ be a non-continuous group homomorphism with dense image.  The maps $\phi _i $ satisfy all the conditions of Definition \ref{def:ga} except \eqref{item:ga-5}.
\end{itemize}

\end{Exa}

\section{Preliminaries}\label{sec:prelims}

\subsection{Notation}
In a topological space $X$, if $A \subset X$, we denote by $\inte (A)$ the interior of $A$, and by $\overline{A}$ the closure of $A$. In a metric space $X$, we denote by $B_r^X(x)$ the open ball with center $x$ and radius $r$, and by $\overline{B}_r^X(x)$ the corresponding closed ball. If the metric space is clear from the context, we write $B_r(x)$ and $\overline{B}_r(x)$ instead.

For a group $G$, we denote its identity element by $e_G$, or by $e$ when the group is clear from the context. We say a subset $S \subset G$ is \emph{symmetric} if it contains the identity and is closed under inverses. When $G$ is a topological group, we write $H \leq G$ if $H$ is a closed subgroup.

If $G$ is a Lie group with Lie algebra $\mathfrak{g}$, we denote by $\exp $ and $\log$ the algebraic exponential and logarithm maps between (subsets of) them, even when $G$ is equipped with a Riemannian metric.

\subsection{Hyperspaces}

Let $X$ be a metric space. We denote by $\mathfrak{M}(X) $ the set of  closed sets of $X$ equipped with the Hausdorff distance $d_H$. 

\begin{Thm}[Blaschke \cite{%burago-burago-ivanov,
petrunin}]\label{thm:blaschke}
  %  If $X$ is compact, then so is $\mathfrak{M}(X)$.
 $X$ is compact if and only if  $\mathfrak{M}(X)$ is compact.
\end{Thm}
    
\begin{Lem}\label{lem:theta-continuity}
    Let $X$ be a metric space and   $\theta : [ a,b] \to \mathfrak{M}(X) $ a function with
    \begin{equation}\label{eq:theta-monotone}
          \theta (r) \subset \theta (s) \text{ for all } r \leq  s .             
    \end{equation}
    If $\theta (b)$ is compact, then $\theta$ has at most countably many discontinuities.
\end{Lem}

\begin{proof}
    Fix $n \in \mathbb{N}$, and let 
    \[      S_{n} : = \{ r \in [a,b] \, \vert \,    \limsup _{t \to r} d_H( \theta (t), \theta (r)  ) > 1/ n \} .                                           \]
    Take $r_1 < r_2 < r_3 $ with $r_2 \in S_{n}$. By definition of $S_{n}$, there is $t \in (r_1, r_3)$ with 
    \[     d _H(\theta (t), \theta (r_2)) \geq 1/n.                \]
    If $t < r_2$, then $d_H(\theta (r_1), \theta (r_2)) \geq 1/n$, and if $t> r_2$, then we have $d_H(\theta (r_3), \theta (r_2)) \geq 1/n$. In either case,  from \eqref{eq:theta-monotone} we have
    \[   d_H(\theta (r_1), \theta (r_3)) \geq 1/n .                    \]
    Therefore, the number of elements of $S_{n} \backslash \{ a, b \} $ cannot exceed the maximal number of $1/n$-separated elements in $\mathfrak{M}(\theta (b))$, which by Theorem \ref{thm:blaschke} is finite. Then the set of discontinuities of $\theta$ is contained in the countable set  $ \bigcup_{n \in \mathbb{N}} S_{n} .  $
\end{proof}

\subsection{Gromov--Hausdorff convergence} 
\begin{Def}[Fukaya \cite{fukaya}]    
Let $(X,p)$ and $(Y,q)$ be pointed proper metric spaces and $G \leq \iso (X)$, $H \leq \iso (Y)$ closed subgroups.  For $\varepsilon > 0 $, an $\varepsilon$-Hausdorff approximation between the pairs $(X,G,p) $ and $ (Y,H,q)$ consists of a metric $\overline{d}$ on $X \sqcup Y$ and functions $\phi : G \to H$ and $\psi : H \to G$ such that 
    \begin{itemize}
        \item The restriction of $\overline{d}$ to either of $X$ or $Y$ coincides with the original one.
        \item For each $x \in B_{1/ \varepsilon} ^X (p)$, there is $y \in B_{1/ \varepsilon}^Y(q)$ with $\overline{d}(x,y) < \varepsilon$. 
        \item For each $y \in B_{1 / \varepsilon} ^Y (q)$, there is $x \in B_{1/ \varepsilon} ^X (p)$ with $\overline{d}(x, y) < \varepsilon$. 
        \item $\overline{d}(p,q) < \varepsilon$. 
        \item For each $x \in B_{1/ 3\varepsilon}(p)$, $y \in B_{1/3 \varepsilon}(q)$,  $g \in G$, and $h \in H$  with $d(gp,p), d(hq,q) < 1/ 3\varepsilon$, one has
        \begin{gather*}
               \vert \overline{d}( x,y ) - \overline{d} (gx , \phi(g) y) \vert < \varepsilon, \\
                \vert \overline{d}( x,y ) - \overline{d} (\psi (h)x , h y) \vert < \varepsilon  .
        \end{gather*}
    \end{itemize}
    The equivariant Gromov--Hausdorff distance $d_{eGH}((X,G,p), (Y,H,q)) $ is defined to be the infimum  of $1/5$ and the $\varepsilon > 0 $ for which there are $\varepsilon$-Hausdorff approximations between the triples $(X,G,p)$ and $(Y,H,q)$.   
\end{Def}

Given a sequence $(X_i,p_i)$ of pointed proper metric spaces and $G_i \leq \iso (X_i)$ closed subgroups, we write 
\[   (X_i,G_i, p_i) \xrightarrow{eGH} (X,G,p)    \]
if $(X,p)$ is a pointed proper metric space, $G \leq \iso (X)$ a closed subgroup, and 
\[     d_{eGH} ( (X_i,G_i,p_i), (X,G,p)) \to 0 \text{ as } i \to \infty .      \]
In such a case, there are sequences of functions $f_i : X_i \to X$ and $\phi_i : G_i \to G$ with $ \lim _{i \to \infty} d(f_i (p_i),p)  = 0 $ and such that:
\begin{comment}
\begin{gather*}
    \lim_{i \to \infty}  \sup _{x,y \in B_r(p_i)}   \vert d(f_i (x), f_i(y)) - d(x,y) \vert  = 0 . \\
    \lim_{i \to \infty} \sup _{y \in B_r(p)} \inf _{x \in B_{2r}(p_i)} d(f_i (x),y) = 0 .\\
    \lim_{i \to \infty }   \sup _{\substack{ g \in G_i \\ d (g p_i, p_i ) < r }} \sup _{x \in B_{r}(p_i)}   d( f_i ( g  x ) , \phi_i (g) f_i (x) )  = 0 . 
\end{gather*}
For each $g \in G$ with $d (gp,p) < r$, there are $g_i \in G_i$ with $d(g_ip_i,p_i ) < r$ and 
\[         \lim_{i \to \infty} \sup _{x \in B_r(p_i)} d(g f_i (x), \phi_i (g_i) f_i (x))  = 0 .       \]    
\end{comment}
\begin{itemize}
    \item For each $r > 0 $, one has
    \[  \lim_{i \to \infty}  \sup _{x,y \in B_r(p_i)}   \vert d(f_i (x), f_i(y)) - d(x,y) \vert  = 0 .   \]
    \item For each $r > 0 $, one has
    \[   \lim_{i \to \infty} \sup _{y \in B_r(p)} \inf _{x \in B_{2r}(p_i)} d(f_i (x),y) = 0 .       \]
    \item  For each $r > 0 $, one has
    \[  \lim_{i \to \infty }   \sup _{\substack{ g \in G_i \\ d (g p_i, p_i ) < r }} \sup _{x \in B_{r}(p_i)}   d( f_i ( g  x ) , \phi_i (g) f_i (x) )  = 0 .   \]
    \item For each $g \in G$, there are $g_i \in G_i$ with $\limsup_{i \to \infty }d(g_ip_i,p_i ) < \infty$ and   
\[         \lim_{i \to \infty} \sup _{x \in B_r(p)} d(g x, \phi_i (g_i) x )  = 0        \]
for each $r > 0 $. 
\end{itemize}
Indeed, if $d_{eGH}((X_i,G_i,p_i), (X,G,p)) < \varepsilon \leq 1/ 5 $ and one has an $\varepsilon$-Hausdorff approximation between the triples $(X_i,G_i,p_i)$ and $(X,G,p)$, then one can define $f_i$ by sending $x \in B_{1/ \varepsilon}(p_i)$ to a point $y \in B_{1/ \varepsilon }(p)$ with $\overline{d} (x,y) < \varepsilon$, and sending $X_i \backslash B_{1/ \varepsilon} (p_i)$ to $p$. The last condition can be verified by setting $g_i : = \psi_i (g)$ with $\psi_i : G \to G_i$ coming from the definition of $\varepsilon$-Hausdorff approximation. 

Conversely, if the triples $(X_i,G_i,p_i)$ and $(X,G,p)$ are such that there are functions $f_i : X_i \to X$, $\phi_i : G_i \to G$ satisfying the above properties, then 
\[ (X_i,G_i,p_i) \xrightarrow{eGH}(X,G,p) . \] 

A feature of this framework is the following compactness result. 

\begin{Thm}[Fukaya--Yamaguchi \cite{fukaya-yamaguchi}]\label{thm:fy}
    Let $(X_i,p_i)$ be a sequence of pointed proper metric spaces, and $G_i \leq \iso (X_i)$ a sequence of closed groups. If the sequence $(X_i,p_i)$ converges in the Gromov--Hausdorff sense to a pointed proper metric space $(X,p)$, then after passing to a subsequence, there is $G\leq \iso (X)$ such that 
    \[      (X_i,G_i,p_i) \xrightarrow{eGH} (X,G,p) .              \]
\end{Thm}

A main ingredient of the above theorem is the following compactness property.

\begin{Pro}[\cite{gromov}]\label{pro:isometry-lemma}
    Let $g_i \in G_i$  be a sequence with 
    \[     \limsup _{i \to \infty} d(g_i p_i , p_i ) < \infty .  \] 
    Then after passing to a subsequence, $\phi_i (g_i)$ converges to an element $g \in G$.
\end{Pro}

The following result establishes that the maps $\phi_i$ are almost distance-preserving.

\begin{Pro}[\cite{zamora}]\label{pro:egh-is-pgh}
    For any pair of sequences $g_i, h_i \in G_i$ with 
    \[     \limsup _{i \to \infty} \max \{  d(g_i p_i , p_i ) , d(h_ip_i, p_i) \} < \infty ,   \] 
    one has 
    \[     \lim_{i \to \infty} \vert d_{p_i} (g_i,h_i) - d_p (\phi_i (g_i), \phi_i (h_i )) \vert  = 0  .    \]
\end{Pro}

The following proposition shows that the maps  $\phi_i$ are almost group homomorphisms. 

\begin{Pro}\label{pro:almost-morphism}
    Let $(X,p) $, $(Y,q)$ be pointed proper metric spaces, $G \leq \iso (X)$, $H\leq \iso (Y)$ closed groups, $\varepsilon \in (0, 1/5 ] $, and $(\overline{d}, \phi , \psi ) $ an $\varepsilon$-Hausdorff approximation between the triples $(X,G,p) $ and $(Y,H,q)$. Then for any $g_1,g_2 \in G$ with $d(g_1p,p), d(g_2p,p)  <  1/ 6 \varepsilon $, one has
    \[      d (  \phi (g_1 g_2 )y, \phi (g_1) \phi(g_2) y   )  < 7 \varepsilon      \]
    for all $y \in B_{1 / 12 \varepsilon }(q)$. In particular, 
    \[      d_q (\phi (g_1g_2), \phi(g_1)\phi(g_2)) < 20 \varepsilon.                \]
\end{Pro}

\begin{proof}
    Pick $x_0 \in B_{1/ \varepsilon }(p)$ with $\overline{d}(x_0,y) < \varepsilon$, and $y' \in B_{1/ \varepsilon }(q)$ with $\overline{d}(g_2 x_0, y') < \varepsilon$. Then
    \begin{align*}
        d (\phi(g_1g_2)y, \phi(g_1)\phi(g_2)y) & \leq  \overline{d} (\phi (g_1g_2)y, g_1g_2 x_0) + \overline{d}(g_1g_2x_0, \phi(g_1) y' ) + \overline{d} (\phi (g_1)y', \phi(g_1 )\phi(g_2)y)\\
        & \leq  \overline{d}(y,x_0) + \overline{d} (g_2x_0, y' ) + \overline{d}(y', \phi(g_2)y) + 2 \varepsilon \\
        & \leq  \overline{d}(y' , g_2x_0) + \overline{d}(g_2x_0, \phi(g_2) y) + 4 \varepsilon \\
        & \leq  \overline{d}(x_0, y) + 6 \varepsilon  \leq 7 \varepsilon .
    \end{align*}
\end{proof}

\begin{Pro}\label{pro:small}
    Let $(X_i,p_i)$ be a sequence of pointed proper metric spaces, and $G_i \leq \iso (X_i)$ a sequence of closed groups. Assume
\[     (X_i, G_i, p_i) \xrightarrow{eGH} (X,G,p)  .                       \]
If $H_i \trianglelefteq G_i$ is a sequence of small normal subgroups, then 
\begin{equation}\label{eq:small-convergence}
    (X_i / H_i , G_i / H_i , [p_i]) \xrightarrow{eGH} (X, G, p),   
\end{equation}  
and for any sequence of small subgroups $K_i \leq G_i / H_i \leq \iso (X_i / H_i )$, the sequence of preimages $\tilde{K}_i \leq G_i$ also consists of small subgroups. 
\end{Pro}

\begin{proof}
    By definition of small subgroups, the $H_i$-orbits in $X_i$ that intersect $B_r(p_i)$ with $r > 0 $ fixed have diameter going to $0$ as $i \to \infty$, so the projections $X_i \to X_i/H_i$ and $G_i \to G_i / H_i $ show that the equivariant Gromov--Hausdorff distance between the triples $(X_i, G_i, p_i)$ and $(X_i/H_i, G_i/H_i, [p_i])$ goes to $0$ as $i \to \infty$, so  \eqref{eq:small-convergence} holds.

    Let $K_i\leq G_i / H_i$ be a sequence of small subgroups, and let $\tilde{K}_i \leq G_i$ be the corresponding preimages. Fix $r>0$ and pick elements $k _i \in \tilde{K}_i $, $x_i \in B_r(p_i)$. By the definition of quotient metric, we have
    \[         d_{X_i}(k_i x_i, x_i) \leq d_{X_i/H_i}( k_i [x_i], [x_i]  ) + \diam _{X_i}[x_i] + \diam _{X_i}[k_i x_i]  , \]
    where $[x_i] \subset X_i$ stands for the $H_i$-orbit of $x_i$.         Since the groups $K_i$ are small, the first summand on the right hand side goes to $0$ as $i \to \infty$, and since the groups $H_i$ are small, the other two summands on the right go to $0$ as well. This shows that the groups $\tilde{K}_i$ are small.
\end{proof}

\subsection{Properties of good approximations}

In this section we cover the technical properties of good approximations. We first note that they are stable under minor perturbations. 

\begin{Obs}\label{obs:perturbation}
    Let $G_i$ be a sequence of locally compact Hausdorff groups and $G$ a locally compact Hausdorff group. Let $A_i \subset G_i $ and $A \subset G$ be open pre-compact subsets and $\phi_i , \phi ' _i : G_i \to G $ two sequences of functions. Assume that for each $n \in \mathbb{N}$ and each open symmetric set $U \subset G$, there is $i_0 \in \mathbb{N}$ such that for $i \geq i_0$ one has 
    \[          \phi _i (g) ^{-1} \phi _i ' (g)  \in U         \]
    for all $g \in A_i ^n$.  If the sequence $\phi_i$ consists of good approximations with regular neighborhoods $A_i $ and $A $, then so does the sequence $\phi_i'$.  
\end{Obs}

The proof is straightforward since the verification of each property in the definition of good approximations for $\phi_i'$ essentially follows from the corresponding one for $\phi_i$.  

The following compactness property is an easy consequence of the definitions and will be used without reference. 

\begin{Lem}\label{lem:convergent-sequences}
    Let $\phi_i : G_i \to G$ be a sequence of good approximations with regular neighborhoods $A_i \subset G_i $ and $ A  \subset G$ and assume $G$ is metrizable. Then for each $n \in \mathbb{N}$ and any sequence $g_i \in A_i^n$, after passing to a subsequence, there is $g \in \overline{A}^n$ with $\phi_i (g_i) \to g$. Moreover, for each $m \in \mathbb{N}$, one has $\phi_i (g_i ^m) \to g^m$. In particular, $\phi_i (e_{G_i} ) \to e_G$. 
\end{Lem}

\begin{proof}
    First consider a sequence $g_i \in A_i$. By \eqref{item:ga-2}, one has $\phi_i (g_i) \in A^2$ for large enough $i$, so after passing to a subsequence, one has $\phi_i(g_i) \to g$ for some $g \in \overline{A}^2$. If $g \notin \overline{A}$, then one can find an open set $V \subset G$ with $\overline{A} \subset V$ but $g \notin \overline{V}$. This would imply $\phi_i (g_i ) \notin V$ for $i$ sufficiently large, contradicting  \eqref{item:ga-2}. 
    
    Now consider a sequence $g_i \in A_i^n$. We can write $g_i = h_{i,1} \cdots h_{i, n}$ with $h_{i,j} \in A_i$ for all $i$ and $j$. By the first part, after passing to a subsequence, there are $h_1, \ldots , h_n \in \overline{A}$ with  $  \phi_i(  h_{i,j}) \to h_j   $ for each $j$. Using \eqref{item:ga-3}, one has
    \[      \phi_i (g_i) \to h_1 \cdots h_n \in \overline{A}^n .  \]
    The fact that $\phi_i (g_i^m) \to g^m$ for every $m$ also follows from \eqref{item:ga-3}, and the last claim about the identity elements follows from the previous one with $m = 2$. 
\end{proof}

The following result is inspired by the third item of \cite[Definition 3.5]{breuillard-green-tao}.

\begin{Lem}\label{lem:approximation-by-internal}
    Let $\phi_i : G_i \to G$ be a sequence of good approximations with regular neighborhoods $A_i \subset G_i $ and $A \subset G$.  For all $K \subset U \subset A$ with $K$ compact and $U$ open, there are open sets $U_i \subset A_i  $ with 
        \[ \phi_i^{-1}(K) \cap A_i \subset U_i \subset \phi_i^{-1}(U)                 \]
        for $i$ large enough.  Moreover, if $U$ is symmetric, the sets $U_i$ can be taken to be symmetric.
\end{Lem}

\begin{proof} 
    Let $W \subset G$ be a symmetric open set with $K W ^2 \subset U $. By \eqref{item:ga-5}, there is a sequence of open symmetric sets $W_i \subset A_i$ with $\phi_i (W_i) \subset W$ for $i$ large enough. Then define   
    \[   U_i : =    ( (  \phi_i ^{-1} (K ) \cap A_i )  W_i  ) \cap A_i .                       \]
    The fact that 
    \[    \phi_i^{-1} (K) \cap A_i \subset U_i     \]
    is clear from the construction. For $k _i \in \phi_i ^{-1} (K)\cap A_i$ and $w_i \in W_i$, we have
    \[  \phi_i (k_i w_i )  = [\phi_i (k_i)] [ \phi_i (w_i )][   \phi _i (w_i) ^{-1} \phi_i (k_i) ^{-1} \phi _i (k_i w_i) ]  .                        \]
    By \eqref{item:ga-3}, the expression on the right-hand side belongs to $K W ^2 \subset U$ for $i$ large enough. This shows $U_i \subset  \phi_i ^{-1} (U)$ for $i$ large enough, finishing the proof of the first part.  
    
    To prove the second part, assume $U$ is symmetric. Pick $V \subset A$ an open symmetric set with 
    \[ K \cup K^{-1} \subset  V \subset \overline{V} \subset U, \] 
    and let $W \subset G$ be an open  symmetric set with 
    \[ W ^2 V  \cup V W  ^2 \subset U . \] 
    By \eqref{item:ga-5}, there is a sequence of open symmetric sets  $W_i \subset A_i$ with $\phi_i (W_i) \subset W$ for $i$ large enough. Set 
    \[   K_i : = \phi_i ^{-1} (K) \cap A_i                  \] 
    and 
    \[    U_i : =  (  W_i ( K_i  \cup K_i ^{-1}  )    \cup  ( K_i  \cup K_i ^{-1}  )  W_i )  \cap       A_i .             \]
    It is clear that $U_i$ is symmetric and  $K_i \subset U_i$. From \eqref{item:ga-3}, it follows that $\phi_i ( K_i \cup K ^{-1}) \subset V$ for $i$ large enough. Then arguing as in the first part, one has
    \begin{align*}
        \phi_i ( W_i (K_i \cup K_i ^{-1}) \cap A_i  ) & \subset W^2 V , \\
        \phi_i ( (K_i \cup K_i ^{-1}) W_i  \cap A_i  ) &   \subset V W^2 .
    \end{align*}
 Consequently,   $   \phi_i (U_i) \subset  U         $    for $i$ large enough, finishing the proof of the second part. \end{proof}

\begin{Rem}
An earlier version of this paper included the first part of Lemma \ref{lem:approximation-by-internal} in place of \eqref{item:ga-5} as part of the definition of good approximations. It is not hard to show that assuming \eqref{item:ga-1}-\eqref{item:ga-4}, these two conditions are equivalent.      
\end{Rem}

The following result allows one to ``localize'' a good approximation.

\begin{Pro}\label{pro:zoom-i}
    Let $\phi_i: G_i \to G$ be good approximations with $A_i \subset G_i$, $A\subset G$ as regular neighborhoods, and assume $G$ is  metrizable. For any open symmetric set $B \subset A$, there are open symmetric sets $B_i \subset A_i$ such that the maps $\phi_i $ are good approximations with regular neighborhoods $B_i$ and $B $. Moreover, a sequence of groups $H_i \leq G_i$ is small with respect to the pairs $(\phi _i ,  A_i)$  if and only if it is small with respect to the pairs $(\phi_i, B_i)$. 
\end{Pro}

\begin{proof}
    Consider an increasing sequence of compact symmetric sets $K_m \subset B $ such that 
    \[  \bigcup_{m = 1}^{\infty} \text{int}(K_m)   = B.  \]
    By Lemma \ref{lem:approximation-by-internal} and a diagonalization argument, there is a sequence of open symmetric sets  $B_i \subset A_i$ such that for any  $m \in \mathbb{N}$, one has
    \begin{equation}\label{eq:b-i-def}
        \phi_i^{-1}(K_m) \cap A_i \subset B_i  \subset \phi_i^{-1}(B)     
    \end{equation}
    for $i$ large enough.  All the properties of good approximation for the pairs $(\phi_i, B_i)$ can be checked directly from the corresponding ones for the pairs $(\phi_i, A_i)$. 

    Finally, consider a sequence of subgroups $H_i \leq G_i$. If they are small with respect to $(\phi_i, B_i)$, then they are clearly small with respect to $(\phi_i, A_i)$. On the other hand, assume they are small with respect to $(\phi_i, A_i)$. Pick  $U \subset A$ an open symmetric set with $\overline{U} \subset B $. By compactness of $\overline{U}$, there is $m \in \mathbb{N}$ with $U \subset K_m$ and by our choice of $B_i$ we conclude that  $   H_i \subset \phi _i^{-1} (U) \cap A_i \subset  B_i  $   for $i$ large enough, showing that the groups $H_i $ are small with respect to the pairs $(\phi_i, B_i)$.     
\end{proof}

\begin{comment}

    Note that the sets $B_i$ constructed in the proof of Proposition \ref{pro:zoom-i} satisfy the following condition:
    \begin{enumerate}[label=IV']
        \item For each $n \in \mathbb{N}$ and any compact $K \subset B$, for $i$ large enough, if $g \in A_i^n$ satisfies $\phi_i (g) \in K$, then $g \in B_i$.  \label{item:ga-4a}
    \end{enumerate}
    This follows from \eqref{eq:b-i-def} because $K \subset K_m$ for some $m$.

\end{comment}

\subsection{Approximately continuous maps}

\begin{Def}
    Let $X$ be a topological space, $Y$ a metric space,  $f: X \to Y$ a function, and $\varepsilon > 0 $. We say $f$ is $\varepsilon$\emph{-continuous} if for any $p \in X$, there is a neighborhood $U \subset X$ of $p$ with  
    \[ d(f(x),f(p)) < \varepsilon \text{ for all }x \in U. \]
\end{Def}

\begin{Rem}\label{rem:continuous}
    It is straightforward to show that $f : X\to Y$ is $\varepsilon$-continuous if and only if for any set $S \subset X$, there is an open set $U \subset X$ containing $S$ with 
    \[ f(U) \subset \bigcup _{s \in S} B_{\varepsilon}(f(s)). \] 
\end{Rem}

\begin{Pro}\label{pro:approx-continuity}
    Let  $\phi_i : G_i \to G$ be a sequence of good approximations with regular neighborhoods $A_i \subset G_i $ and $A \subset G$. For each  $n \in \mathbb{N}$ and each open symmetric set  $U \subset G$,  there is $i _ 0 \in \mathbb{N} $  such that if   $i\geq i_0$, then each $p \in A_i^n$ has a neighborhood $V \subset G _ i $ with 
    \[      \phi_i (x)^{-1} \phi _i (p) \in U  \text{ for all } x \in V .                   \]
\end{Pro}

\begin{proof}
    Let $W \subset G $ be an open symmetric set with $W^2 \subset U$.   By \eqref{item:ga-5}, there is a sequence of open symmetric sets $W_i \subset A_i$ with $\phi_i (W_i)  \subset W$ for $i$ large enough. For $p \in A_i ^n$ and $w \in  W_i $, we have 
    \[   \phi_i (pw) ^{-1} \phi _i (p) = [\phi_i (pw) ^{-1} \phi_i (p) \phi_i (w) ] [\phi_i (w) ^{-1}] .    \]
    By \eqref{item:ga-3}, the expression on the right belongs to $W^2 \subset U$ for $i$ large enough, depending only on $n$ and $W$. Hence $V : = p W_i$ is the desired neighborhood.
\end{proof}

\begin{Cor}\label{cor:approx-continuity}
    Let  $\phi_i : G_i \to G$ be a sequence of good approximations with regular neighborhoods $A_i \subset G_i $ and $A \subset G$. If $G$ is equipped with a left-invariant metric, then for each $n \in \mathbb{N}$ and each $\varepsilon > 0$, the map $\phi_i$ restricted to $A_i^n$ is $\varepsilon$-continuous for $i$ large enough.
\end{Cor}

\begin{proof}
    Apply Proposition \ref{pro:approx-continuity} with $U = B_{\varepsilon}(e)$. 
\end{proof}

A continuous function defined on a compact metric space is uniformly continuous. A similar statement holds for $\varepsilon$-continuous functions. 

\begin{Pro}\label{pro:uniform-continuity}
    Let $X, Y$ be metric spaces with $X$ compact, $\varepsilon > 0 $, and $f: X \to Y$ an $\varepsilon$-continuous function. Then there is $\delta > 0 $ such that for all $x_1, x_2 \in X$ with $d(x_1, x_2 ) < \delta$, one has $d(f(x_1), f(x_2)) < 2 \varepsilon$. 
\end{Pro}

\begin{proof}
Assuming the contrary, there would be a sequence of pairs $(x_1^i, x_2^i) $ in $X$ with $d(x_1^i, x_2^i) \to 0$  as $i \to \infty$, but $d(f(x_1^i), f(x_2^i)) \geq 2 \varepsilon$ for all $i \in \mathbb{N}$. Since $X$ is compact, after passing to a subsequence, we would have $x_1^i \to p$ for some $p \in X$. From the definition of $\varepsilon$-continuity, for $i$ large enough we would have
\[   d(f(x_1^i ), f(x_2 ^i))  \leq d (f(x_1^i), f(p)) + d(f(p) , f(x_2 ^i)) < 2 \varepsilon           , \]
which is a contradiction.
\end{proof}

\begin{Thm}[Approximate Borsuk--Ulam]\label{thm:abu}
If  $n > k$, for any $\varepsilon$-continuous function 
\[ f : \mathbb{S}^{n-1} \to \mathbb{R}^{k}\] 
with $f(-x ) = - f(x)$ for all $x \in \mathbb{S}^{n-1}$, there is $x_0 \in \mathbb{S}^{n-1}$ with $f(x_0) \in  B_{ 2 \varepsilon } (0) .$
\end{Thm}

\begin{proof}
     By Proposition \ref{pro:uniform-continuity}, there is $\delta > 0 $ such that if $x,y \in \mathbb{S}^{n-1}$ satisfy $d(x,y) < \delta$, then $d(f(x), f(y)) < 2 \varepsilon$. Put a triangulation on $\mathbb{S}^{n-1} $ that is invariant under the transformation $x \mapsto -x$ and such that each simplex has diameter $< \delta$. Define 
    \[ \tilde{f}: \mathbb{S}^{n-1} \to \mathbb{R}^{k} \] 
    in such a way that agrees with $f$ on the $0$-skeleton of the triangulation, and is affine when restricted to each simplex. Then $\tilde{f}$ is continuous, $\tilde{f}(-x) = - \tilde{f}(x)$ for all $x \in \mathbb{S}^{n-1} $, and $\vert \tilde{f}(x) - f(x) \vert < 2 \varepsilon$ for all $x \in \mathbb{S}^{n-1}$. By the Borsuk--Ulam Theorem, there is $x_0  \in \mathbb{S}^{n-1}$ with $\tilde{f}(x_0 ) = 0$, and consequently $f(x_0 ) \in B_{2\varepsilon} (0)$.
\end{proof}

\subsection{Escape norm}\label{sec:escape}

Let $G $ be a group and $A \subset G$ an open symmetric set. The \emph{escape norm} $\Vert \cdot \Vert _{A} : G \to \mathbb{R}$  is defined as
\[    \Vert g \Vert _A : =  \inf \left\{ \frac{1}{m+1} \,\, \Big| \,\, g^j \in A  \text{ for all } j \in \{ 0, 1, \ldots , m \} \right\}  .                \]
This norm is often not continuous as it attains only a discrete set of values (outside of $0$). However, it is easy to see that for any convergent sequence $g_j \to g $ one has
\[     \limsup_{j \to \infty} \Vert g_j \Vert_{A} \leq \Vert g \Vert _{A}  .                       \]
In the other direction, we have the following result.

\begin{Lem}\label{lem:norm-continuity}
    Let $\phi_i : G_i \to G $ be a sequence of good approximations with regular neighborhoods $A_i \subset G_i$ and $A \subset G $. Assume $G$ is metrizable and $A$ has the property that if $h \in \overline{A}^3 \backslash A $, then $h^2 \notin \overline{A}$. Then for large enough $i$, and any convergent sequence $g_j \to g $ in $G_i$, one has
    \begin{equation}\label{eq:norm-bad-continuity}
        \Vert g \Vert_{A_i} \leq 2 \liminf_{j \to \infty} \Vert g_j \Vert _{A_i} .
    \end{equation}
\end{Lem}

\begin{proof}
    Assuming the result fails, after passing to a subsequence there are  $g_i, g_{i,j} \in G_i$  such that $g_{i,j} \to g_i$ for each $i$, but \eqref{eq:norm-bad-continuity} fails. 
%    \[
%         \Vert g _i \Vert _{A_i} > 2 \lim_{j \to \infty} \Vert g_{i,j} \Vert _{A_i} .
%    \]
    Since outside of $0$, $\Vert \cdot \Vert_{A_i}$ attains only a discrete set of values, we can further assume \begin{equation}\label{eq:norm-bad}
     \Vert g_i \Vert _{A_i} > 2 \, \Vert g_{i,j} \Vert _{A_i} \text{ for all } j\in \mathbb{N}  . \end{equation}   
    Notice that  $g_{i,j}\in A_i$ for all $j$, otherwise the right hand side of \eqref{eq:norm-bad} equals 2, which is impossible.  This implies that $g_i \in \overline{A}_i \subset A_i^2$ for each $i$. Set 
\[     m_i : =  \frac{ 1 }{\Vert g_i \Vert _{A_i}}   .             \]
Notice that $g_i ^{m_i} = g_i ^{m_i - 1} g_i \in A_i^3$ for all $i$, so by Lemma \ref{lem:convergent-sequences}, after passing to a subsequence, one has 
\[   \phi_i (g_i ^{m_i }) \to h \]
for some $ h  \in \overline{A}^3$. By Proposition \ref{pro:approx-continuity}, there is a sequence $n_i \in \mathbb{N}$ with $\phi_i (g_{i,n_i}^{m_i}) \to h$. 

Notice that $ h \notin A $. Otherwise,  by \eqref{item:ga-4} one would have $g_i^{m_i} \in A_i$ for large $i$, contradicting the definition of $m_i$.   Finally, notice that $\phi_i (g_{i,n_i}^{2m_i}) \to h^2 \notin  \overline{A} $, so by \eqref{item:ga-2}, one has $g_{i,n_i} ^{2m_i} \notin A_i$ for $i$ large enough, meaning that
\[     \Vert g_{i,n_i } \Vert _{A_i} \geq \frac{1}{2m_i} = \frac{\Vert g_i \Vert_{A_i}}{2},            \]
    which contradicts \eqref{eq:norm-bad}. 
\end{proof}

\begin{Exa}
 Let $G : = \mathbb{S}^1 $,  $A : =  \mathbb{S}^1 \backslash \{ -1 \} $, $g = -1 $, and $g_j : = e^{( \pi - 1/ j ) \sqrt{-1} } $.  Then $g_j \to g$, but 
 \[      \Vert g \Vert _A = 1 , \hspace{2cm}  \Vert g_j \Vert _{A}  = 0    \]
for all $j$, showing that the condition $h \in \overline{A}^3 \backslash A \Rightarrow h^2 \notin \overline{A}$ in Lemma \ref{lem:norm-continuity} was necessary. 
\end{Exa}

If $G$ is a Lie group with Lie algebra $\mathfrak{g}$, and $B \subset G$ is an open symmetric set, we can also define $\vert \cdot \vert _B  : \mathfrak{g} \to \mathbb{R}$ as:
\[  \vert \vv \vert _B :   = \frac{1}{\tau_B (\vv)}  ,   \]
where $\tau_B (\vv) : = \inf \{ t > 0  \vert \exp (t \vv) \notin B \} \in (0, \infty ]$.   While $\vert \cdot \vert _B$ is  homogeneous, it is in general not a semi-norm due to the failure of the triangle inequality. Nevertheless,  in certain situations it can be approximated by a legitimate norm (see Lemma \ref{lem:norm}). We now show its relationship with the escape norm.
\begin{Pro}\label{pro:norm-definition}  
For all $\vv \in \mathfrak{g}$ one has 
    \begin{equation}\label{eq:escape-less-than-norm}
         \Vert \exp (\vv) \Vert _{B}     \, \leq  \, \vert \vv \vert _{B}   
    \end{equation}
    and 
    \begin{equation}\label{eq:norm-definition}
        \vert \vv \vert _B  = \limsup_{t \to 0} \frac{\Vert \exp (t\vv ) \Vert _{B}   }{\vert t \vert }  . 
    \end{equation}
\end{Pro}
\begin{proof} To prove \eqref{eq:escape-less-than-norm} we can assume  $\Vert \exp (\vv) \Vert _{B} > 0 $, so there is $m \in \mathbb{N}$ with $\Vert  \exp (\vv) \Vert _{B}  = \frac{1}{m}$ and 
    \[  \exp (m \vv)  =  \exp (\vv ) ^m   \notin B.                        \]
    This implies $\tau _B (\vv) \leq m$, and \eqref{eq:escape-less-than-norm} follows from the definition of $\vert \vv \vert _B $. Since $\vert \cdot \vert _B$ is homogeneous, from \eqref{eq:escape-less-than-norm} we deduce 
    \[   \limsup _{t \to 0}   \frac{\Vert \exp (t \vv ) \Vert _B }{ \vert t \vert } \leq \limsup _{t \to 0} \frac{\vert t \vv \vert _{B}}{ \vert t \vert } = \vert \vv \vert _B .    \]
    To prove the other inequality of \eqref{eq:norm-definition},  we can assume $\tau _B (\vv) < \infty$.  Since $B$ is open, 
\[
\exp ( \tau _B(\vv) \vv ) \notin B  ,
\]
so by the definition of escape norm, one has 
    \[  \Vert \exp (\tau _B(\vv) \vv / m ) \Vert_{B} \geq \frac{1}{m} . \]
    Hence
    \[     \limsup_{t \to 0} \frac{\Vert \exp (t\vv ) \Vert _{B}   }{\vert t \vert }  \geq    \limsup_{ m \to \infty}   \frac{\Vert \exp (\tau _B (\vv)\vv/m ) \Vert _{B}   }{ \tau_B (\vv) / m  } \geq \frac{1}{\tau _B (\vv)}  = \vert \vv \vert _B.  \]
\end{proof}

\section{From eGH convergence to good approximations}\label{sec:egh-to-ga}

In this section we show that the maps given by the definition of equivariant Gromov--Hausdorff convergence are good approximations.

\begin{proof}[Proof of Theorem \ref{thm:egh-to-ga}]
    Let $\theta : [0 , \infty) \to \mathfrak{M}(G) $ be given by 
    \[    \theta (r) : =  \{ g \in G \vert d(gp,p) \leq r  \} . \]
    By Lemma \ref{lem:theta-continuity}, we can choose $r_0 > 0 $ such that $\theta$ is continuous at $r_0$, and set
    \begin{align*}
          A  & : =  \{ g \in G \vert d (gp,p) < r_0 \} , \\
          A_i &  : =  \{ g \in G_i \vert d (gp_i,p_i) < r_0 \} . 
    \end{align*}
    Note that all $A_i$ and $A$ are open and pre-compact.

    From the properties of equivariant Gromov--Hausdorff convergence, for any $g \in A$, there is a sequence $g_i \in G_i$ with
    \[    \lim _{i \to \infty} d(g_ip_i, p_i) = d (gp,p) < r_0                \]
    and 
    \[ \phi_i (g_i) \to g. \]  
    This implies that $d(g_ip_i,p_i)< r_0$ for large enough $i$, and hence $g_i \in A_i$, proving \eqref{item:ga-1}. 

    Now pick $V \subset G$ an open set with $\overline{A}  \subset V$. Assuming \eqref{item:ga-2} fails, after passing to a subsequence, there would be $g_i \in A_i$ such that $\phi_i (g_i) \notin V$ for all $i$. After further passing to a subsequence, we can assume $\phi_i (g_i) \to g \notin \overline{A}$. However, since $d(g_ip_i,p_i) < r_0$ for all $i$, we have $d(gp,p) \leq r_0$, so $g \in \theta (r_0)$. By our choice of $r_0$, for each $\varepsilon > 0 $ there are $\delta > 0 $ and $h \in \theta (r_0 - \delta) \subset A $ with $d_p (g,h) \leq \varepsilon $. This shows $g \in \overline{A}$, a contradiction. Therefore \eqref{item:ga-2} holds. Condition  \eqref{item:ga-3} follows easily from Proposition \ref{pro:almost-morphism}.

    Pick a compact $K \subset A$. If \eqref{item:ga-4} fails, after passing to a subsequence, there would be $n  \in \mathbb{N}$ and  $g_i \in A_i ^n$ with $\phi_i (g_i) \in K$, but $g_i \notin A_i$ for all $i$. After further passing to a subsequence, we can assume $\phi_i (g_i) \to g \in K \subset A$. This implies
    \[   \limsup_{i \to \infty } d(g_ip_i,p_i) = \limsup _{i \to \infty} d(\phi_i (g_i)p,p) = d(gp,p) < r_0,                   \]
    so $d(g_ip_i,p_i) < r_0 $ for all large $i$ and $g_i \in A_i$. This is a contradiction,  showing \eqref{item:ga-4}.

    Finally, to verify \eqref{item:ga-5},   let  $ U \subset G$ be an identity neighborhood. Pick   $\varepsilon > 0  $ so that 
    \[   \{ g \in G \, \vert \, d(gx,x) \leq 2  \varepsilon  \text{ for all } x \in B_{ 1 / 2 \varepsilon}(p) \} \subset U                 \]
    and set 
    \[    U_i : = \{ g \in G_i \, \vert \, d (g x , x ) \leq \varepsilon \text{ for all } x \in B_{1/ \varepsilon }(p_i) \}   .            \] 
    It follows from the properties of $\phi_i$ that $\phi_i (U_i) \subset U$ for $i$ large enough, proving  \eqref{item:ga-5}. 

    Now pick $H_i \leq G_i $ a sequence of subgroups and assume they are small in the sense of Definition \ref{def:ss1}. This means that for any sequence $h_i \in H_i$ one has
    \[      \lim_{i \to \infty} d_{p_i} (h_i, e ) =  0 ,      \] 
    which implies 
    \[   \lim_{i \to \infty} d(h_ip_i,p_i) = 0  ,        \]
    so $h_i \in A_i$ for $i$ large enough, and also 
    \[       \lim_{i \to \infty} d_p (\phi_i (h_i), e) = 0 ,                \]
    showing that the groups $H_i$ are small in the sense of Definition \ref{def:ss2}. 

    Now assume the sequence $H_i$ is small in the sense of Definition \ref{def:ss2}. For any sequence $h_i \in H_i$, one has 
    \[  \limsup_{i \to \infty} d (h_ip_i,p_i)  \leq  r_0 , \]
     so Proposition \ref{pro:egh-is-pgh} implies
    \[ 
        \lim_{i \to \infty } d_{p_i}(h_i,e)  =  \lim_{i \to \infty}d_p(\phi_i (h_i), \phi_i (e))   =  0,
    \]
    showing that the sequence $H_i$ is small in the sense of Definition \ref{def:ss1}.
\end{proof}

\section{Blowing up approximations}\label{sec:blowup}

In this section, given a sequence of good approximations $\phi _i : G_i \to G $ with $G$ a Lie group, we produce a sequence of good approximations $\psi_i : G_i \to \mathbb{R}^k$ with $k : = \dime (G)$ in such a way that the small subgroups with respect to $\psi _i $ are also small with respect to  $\phi_i$.

\begin{Lem}\label{lem:zoom-ii}
    Let $\phi_i: G_i \to G$ be good approximations with $A_i \subset G_i$, $A\subset G$ as regular neighborhoods, and assume $G$ is a Lie group of dimension $k$. Then there is a sequence of good approximations $\psi _i : G_i \to \mathbb{R}^k$ with regular neighborhoods $B_i \subset A_i$ and $B = B_1 (0)$. 

    Moreover, this can be done in such a way that any sequence $H_i \leq G_i$ of small subgroups with respect to the pairs $(\psi_i, B_i)$ is also small with respect to the pairs $(\phi_i, A_i)$.
\end{Lem}

Recall that when we are dealing with a Lie group $G$ with Lie algebra $\mathfrak{g}$, the notations $\exp : \mathfrak{g} \to G$ and $\log: U \to \mathfrak{g}$ refer to the algebraic exponential and logarithm maps, and not the Riemannian ones.

\begin{proof}
    Equip the identity connected component of $G$ with a left-invariant Riemannian metric, and identify its Lie algebra $\mathfrak{g}$ with $\mathbb{R}^k$ via a linear isometry.     For each $m \in \mathbb{N}$ with $B_{1/m}(e_G) \subset A $, let $i_m \in \mathbb{N}$ be such that for all $i \geq i_m$ one has:
    \begin{itemize}
%        \item $d(\phi_i (e_{G_i}), e_G) \leq 1/m^2$.
        \item $\phi_i (A_i)$ intersects each ball of radius $1/m^2$ in $B_{2/m}(e)$. 
        \item $\phi_i \vert _{A_i^2}$ is $1/m^2$-continuous.
        \item For all $g,h \in A_i^2$, one has
        \begin{equation}\label{eq:almost-morphism-blow-up}
             d(\phi_i(gh) , \phi_i(g)\phi_i(h))  \leq 1/ m^2 .                       
        \end{equation} 
    \end{itemize}
By Lemma \ref{lem:approximation-by-internal}, we can assume that for all $i \geq i_m$ there is an open symmetric set $B_i \subset A_i$ with 
        \begin{equation}\label{eq:bi-small}
              \phi_i ^{-1} (B_{1/m} (e_G) ) \cap A_i \subset B_i \subset \phi_i ^{-1}( B_{1/m + 1/m^2} (e_G))   .          
        \end{equation}
    Inductively, we can arrange such that $i_{m} < i_{m+1}$ for all $m$. Fix  $i \in [ i_m , i _{m+1}) $ and  define $\psi_i : G_i \to \mathbb{R}^k$ as follows:
    \[    \psi_i (g) : = \begin{cases}
        m \cdot \log (\phi_i (g)) &\text{ if defined}.\\
        0 & \text{ if } \log (\phi_i (g)) \text{ is not defined.}        
    \end{cases}                     \]
    Note that for fixed $\varepsilon \in ( 0 , 1 / 2 ] $, if  $m \in \mathbb{N}$ is large enough, then
    \begin{equation}\label{eq:exp-close-to-dentity}
    \begin{gathered}
         (1- \varepsilon)  \vert \vv - \ww \vert \leq  d(\exp (\vv), \exp (\ww))  \leq (1 + \varepsilon )    \vert \vv - \ww \vert  , \\
            d(\exp (\vv + \ww ) , \exp (\vv) \exp (\ww ))  \leq  \varepsilon / m .  
    \end{gathered}
    \end{equation}  
    for all $\vv , \ww \in B_{1/ \varepsilon m }(0) \subset \mathfrak{g}$. For such $m$ and $i \in [ i_m , i _{m+1}) $,  $\psi_i (B_i)$ intersects each ball of radius $2/m$ in $B : = B_1 (0) \subset \mathfrak{g}$, so \eqref{item:ga-1} holds.    \eqref{item:ga-2} is obtained by combining  \eqref{eq:bi-small} and \eqref{eq:exp-close-to-dentity}. \eqref{item:ga-3} follows from \eqref{eq:almost-morphism-blow-up} and \eqref{eq:exp-close-to-dentity}. 

    \begin{center}
        \textbf{Claim: }For each $n \in \mathbb{N}$, if $i$ is large enough, then $B_i ^n \subset A_i$. 
    \end{center}
    Take $K_0 \subset A$ a compact set that contains a neighborhood of the identity, and pick $m \in \mathbb{N}$ such that $B_{2n/m} (e_G) \subset K_0$.  Take $i \geq  i_m $ with the property that any $g \in A_i ^2 $ with $\phi_i (g) \in K_0$ satisfies $g \in A_i$. Since $\phi_i (B_i ^n ) \subset K_0$,  by induction on $\ell$, we see that $B_i^{\ell} \subset A_i$ for all $\ell \in \{ 1, \ldots , n \}$, proving the claim.   

    To see \eqref{item:ga-4}, fix $K \subset B$ compact and $n \in \mathbb{N}$. Let $\delta > 0 $ be such that $K \subset B_{1- \delta }(0)$, and pick $m_0 \in \mathbb{N}$  so that  \eqref{eq:exp-close-to-dentity} holds with $\varepsilon = \delta / 2 $ for all $m \geq m_0$, and $B_i^n \subset A_i$ for all $i \geq i_{m_0}$.  For $i \in [i_m , i_{m+1})$ with $m \geq m_0$, if $g \in B_i^n$ and $\psi_i (g) \in K$, then $\phi_i (g) \in B_{1/m}(0)$, so $g \in B_i$.

From \eqref{eq:almost-morphism-blow-up}, we deduce $d(\phi_i (e_{G_i}), e_G) \leq 1/m^2$.   From \eqref{eq:exp-close-to-dentity} and the fact that $\phi_i \vert _{A_i ^2} $ is $1/ m^2$-continuous for $i \geq i_m $, we deduce that  $\psi_i \vert _{B_i} $ is $2 / m$-continuous  for $i$ large enough.  This implies \eqref{item:ga-5}.  

    If $i \in [i_m, i_{m+1})$, for any subgroup $H_i \leq G_i$ with $H_i \subset B_i$ one has 
    \[   \phi_i (H_i ) \subset B_{ 2/ m} (e_G) ,                 \]
    so if a sequence of subgroups $H_i \leq G_i$ is small with respect to the pairs $(\psi_i, B_i)$, then it is small with respect to the pairs $(\phi_i, A_i)$. 
\end{proof}

\section{Maximal small subgroups}\label{sec:max-small}

In this section we prove Theorem \ref{thm:largest-small} by establishing the existence of maximal small subgroups. Our argument follows the proof of   \cite[Proposition 9.2]{breuillard-green-tao}.

\begin{Thm}[Breuillard--Green--Tao \cite{breuillard-green-tao}]\label{thm:gleason}  
    Let $\phi_i : G_i \to G$ be a sequence of good approximations with regular neighborhoods $A_i \subset G_i$ and $A \subset G$. Assume $G$ is a Lie group and equip it with a left-invariant Riemannian metric. Set  $B : = B (r)  : = \exp (B_r(0)) \subset G$ and let $B_i \subset G_i$ be the open symmetric sets given by Proposition  \ref{pro:zoom-i}. Then there is $C_0 > 0 $ such that if $r > 0 $ is small enough, for sufficiently large $i$ one has:
    \begin{itemize}
        \item $ \Vert g_1 \cdots g_m \Vert _{B_i} \leq C_0 \sum_{j =1}^m \Vert g_j \Vert _{B_i}   $ for all $ g_1, \ldots , g_m \in G_i.$
        \item $\Vert ghg^{-1} \Vert_{B_i} \leq 1000 \, \Vert h \Vert _{B_i} $ for all $g,h \in B_i^{10}$.
        \item $ \Vert [g,h] \Vert _{B_i}  \leq C_0 \Vert g \Vert _{B_i} \Vert h \Vert _{B_i} $ for all $g,h \in B_i ^{10}$.
    \end{itemize}
\end{Thm}

\begin{proof}
      By \cite[Theorem 8.1]{breuillard-green-tao}, one only needs to show that for $i$ large enough, the sets $B_i$ are strong $K$-approximate groups (as defined on \cite[Definitions 1.2 and 7.1]{breuillard-green-tao}). This is essentially the content of \cite[Proposition 7.3]{breuillard-green-tao}, since taking the ultralimit of the maps $\phi_i : B_i^8 \to G $ produces a good model in the sense of \cite[Definition 3.5]{breuillard-green-tao} with the word ``finite'' in item (iii) replaced by ``open pre-compact'' (this is actually the only point in this paper where we use condition \eqref{item:ga-1}). 

    Notice that while in \cite{breuillard-green-tao} they work exclusively with finite approximate groups, all the relevant material \cite[Definitions 1.2, 3.5, and 7.1,  Proposition 7.3, and Theorem 8.1]{breuillard-green-tao}, works equally well if one uses open pre-compact subsets of locally compact Hausdorff groups instead of finite subsets of local groups. This change also requires that one uses a left-invariant Haar measure instead of the counting measure in the proof of \cite[Theorem 8.1]{breuillard-green-tao}.
\end{proof}

The term \emph{no small subgroup} originated from the following result. It essentially states that if a group doesn't contain small subgroups, then it is a Lie group.

\begin{Thm}[Gleason--Yamabe \cite{yamabe}]\label{thm:gleason-yamabe}
Let $G$ be a locally compact Hausdorff group. If there is  an open symmetric set that contains no non-trivial subgroup, then $G$ is a Lie group.
\end{Thm}

The following is a more detailed version of Theorem \ref{thm:largest-small} (cf. \cite[Propositions 8.5 and 9.2]{breuillard-green-tao}).

\begin{Thm}\label{thm:largest-small-ii}
    Let  $\phi _i : G_i \to G$ be a sequence of good approximations with regular neighborhoods $A_i \subset G_i$ and $A \subset G$. Assume $G$ is a Lie group, let $B_i \subset G_i$, $B (r) \subset G$ be the sets given by Theorem \ref{thm:gleason}, and  let  $H_i  : = \{ g \in G_i \vert \Vert g \Vert _{B_i} = 0 \} $. If $r > 0 $ is small enough, then
    \begin{itemize}
        \item $H_i$ is a subgroup of $G_i$ for $i$ sufficiently large.
        \item The sequence $H_i$ is small.
        \item $H_i \trianglelefteq \langle A_i \rangle $ for $i$ large enough.
        \item For any sequence $H_i ' \leq G_i$ of small subgroups, one has $H_i ' \leq H_i$ for $i$ large enough. 
        \item $G_i ' : = \langle A_i \rangle / H_i$ is a Lie group for $i$ large enough.
    \end{itemize}
\end{Thm}

\begin{proof}
    The fact that $H_i$ is a subgroup follows immediately from Theorem \ref{thm:gleason}.    Consider a sequence $h_i \in H_i$. We claim that $ \phi_i (h_i) \to e_G $.    Otherwise, after passing to a subsequence, we would have $\phi_i (h_i) \to h  $ for some $h \in \overline{B} \backslash \{ e_G \}$. If $r$ was chosen small enough,  $h ^m \notin \overline{B} $ for some $m \in \mathbb{N}$, and $\phi_i (h_i ^m) \to h^m $, so $h_i ^m \notin B_i$ for $i$ large enough, meaning $\Vert h_i \Vert _{B_i} \neq 0$, a contradiction. 

    To show that $H_i \trianglelefteq \langle A_i \rangle$, we need to show that $A_i$ normalizes $H_i$ for $i$ large enough. If this is not the case, after passing to a subsequence, we can find sequences $a_i \in A_i$, $h_i \in H_i$ such that $\Vert a_i h_i a_i^{-1} \Vert_{B_i} \neq 0$ for all $i$. After further passing to a subsequence, $\phi_i (a_i ) \to a \in \overline{A}$. Let $m_i : = 1 / \Vert a_i h_i a_i ^{-1} \Vert _{B_i}$. Since $h_i^{m_i}\in H_i$ and the groups $H_i$ are small, we have
    \[ \lim_{i \to \infty } \phi_i ((a_i h_i a_i^{-1})^{m_i}) =    \lim_{i \to \infty } \phi_i (a_i h_i ^{m_i}a_i^{-1})  =ae_Ga^{-1}=e_G.             \]
    This implies  $  (a_i h_i a_i^{-1})^{m_i} \in B_i$ for $i$ large enough, contradicting the definition of $m_i$.

    Let $H_i ' \leq G_i $ be a sequence of small subgroups. If $h \in  H_i' \backslash H_i $, then $h^m \in H_i' \backslash B_i$ for $m = 1/ \Vert h \Vert _{B_i}$, but by definition $H_i ' \subset B_i$ for all but finitely many $i$'s. 

    By Theorem \ref{thm:gleason-yamabe}, to see that $G_i'$ is a Lie group, it is enough to verify that  the open set $B_i H_i \subset G_i'$ does not contain any non-trivial subgroup for $i$ large enough. If this fails, after passing to a subsequence, there would be $g_i \in B_i ^2 \backslash H_i $ with $g_i ^m \in B_i ^2 $ for all $m \in \mathbb{N}$. Let $m_i : = 1 / \Vert g_i \Vert _{B_i}$. After passing to a subsequence, we have $\phi_i (g_i ^{m_i}) \to g \in \overline{B}^2 \backslash B$. We also have $\phi_i (g^{3m_i}_{i}) \to g^3 \in \overline{B}^2$. However, if $r>0$ was chosen small enough, $g \in \overline{B}^2 \backslash B$ implies  $g^3 \notin \overline{B}^2$.  This is a contradiction, so $G_i'$ is a Lie group for $i$ large enough. 
\end{proof}

\begin{proof}[Proof of Theorem \ref{thm:largest-small-metric}]
    By Theorem \ref{thm:egh-to-ga}, the maps $\phi_i : G_i \to G$ given by the definition of equivariant Gromov--Hausdorff convergence are good approximations.  By Theorem \ref{thm:largest-small-ii}, there is a sequence of small subgroups $H_i \leq G_i  $ with the property that any other sequence of small subgroups $H_i ' \leq G_i$ satisfies $H_i ' \leq H_i$ for $i$ large enough.

    Now assume \eqref{eq:boundedly-generated} holds. By setting $r_0 > R$ in the proof of Theorem \ref{thm:egh-to-ga}, we can guarantee that $A_i$ generates $G_i$ for $i$ large enough, so again by Theorem \ref{thm:largest-small-ii}, $H_i$ is normal in $G_i$ and  $G_i/H_i$ is a Lie group for $i$ large enough. 
    
    If $K_i \leq G_i/H_i$ is a sequence of small subgroups, by Proposition \ref{pro:small}, the preimages $\tilde{K}_i \leq G_i$ are also small, so $\tilde{K}_i \leq H_i$ for $i$ large enough, and $K_i$ is trivial. This shows that the groups $G_i / H_i $  have the NSS property.  Again by Proposition \ref{pro:small}, we have
    \[       (X_i / H_i , G_i / H_i , [p_i] ) \xrightarrow{eGH} (X,G,p) ,                                    \]
    so by Theorem \ref{thm:model} the groups $G_i/H_i$  have dimension $\leq k$ for $i$ large enough. 
\begin{comment}
    let $K_i \leq G_i / H_i \leq \iso (X_i / H_i ) $ be a sequence of small subgroups and let $\tilde{K}_i : = K_i H_i \leq G_i$. 
    
    By the definition of the quotient metric, for any $[x] \in  X_i / H_i $ and any $k \in \tilde{K}_i$ one has 
    \[     d(kx,x) \leq d( k [x], [x] ) + 2 \diam _{X_i}[x] .                          \]
    Fix $r > 0 $ and assume $ x \in B_r(p_i)$. As $i \to \infty$, the terms on the right go to zero since both groups $K_i \leq G_i / H_i$ and $H_i \leq G_i$ are small. This implies that 
    \[           \lim _{i \to \infty} \sup _{k \in \tilde{K}} \sup _{x \in B_r(p)} d(kx,x) = 0           ,               \]
    meaning that the groups $\tilde{K}_i$ are small. This implies that $\tilde{K}_i \leq H_i$ for $i$ large enough, meaning that $K_i$ is trivial. This proves that the groups $G_i / H_i$ have the NSS property.    
\end{comment}
\end{proof}

\begin{Rem}
    Note that the reason $r$ needs to be small in Theorems \ref{thm:gleason} and \ref{thm:largest-small-ii} is to guarantee the accuracy of the first order truncation of the Baker--Campbell--Hausdorff formula, so if $G = \mathbb{R}^k$, then $r$ does not need to be small. 
\end{Rem}

\section{Good approximations and dimension}\label{sec:dimension}

In this section we prove Theorem \ref{thm:main}. We begin with the version for good approximations.

\begin{proof}[Proof of Theorem \ref{thm:model}]
Let $\psi_i : G_i \to \mathbb{R}^k$ be the good approximations given by Lemma \ref{lem:zoom-ii}, and let $B_i\subset G_i$, $B \subset \mathbb{R}^k$ be the corresponding regular neighborhoods. 

\begin{Rem}\label{rem:symmetric}
For any sequence $g_i \in B_i$, we have 
\[   \lim _{i \to \infty} ( \psi_i (g_i ) + \psi _i (g_i^{-1})  ) = 0    ,  \]
so after slightly adjusting the functions $\psi_i $, we can assume $\psi_i (g^{-1}) = \psi _i (g)^{-1}$ for all $g \in B_i$ (see Observation \ref{obs:perturbation}). Henceforth, we enforce this assumption.    
\end{Rem}

By Theorem \ref{thm:largest-small-ii}, the sets $H_i : = \{ g \in G_i \vert \Vert g \Vert _{B_i} = 0  \} $ form a sequence of small subgroups, which by hypothesis are eventually trivial. This implies that  for $i$ large enough one has
\[         \Vert g \Vert _{B_i} > 0   \text{ for all } g \in G_i \backslash \{ e \}.                    \]
Let $\mathfrak{g}_i$ the Lie algebra of $G_i$ and  define $\vert \cdot \vert _{B_i} : \mathfrak{g}_i \to \mathbb{R}$  and $\tau _{B_i} : \mathfrak{g}_i \to ( 0,  \infty ] $ as in Section \ref{sec:escape}. Recall that since $B_i$ is open, we have 
\begin{equation}\label{eq:tau-notin}
\exp ( \tau_{B_i} (\vv) \vv ) \notin B_i      
\end{equation}
for all $\vv \in \mathfrak{g}_i $ with $\vert \vv \vert _{B_i } \neq 0 $. 
\begin{Lem}\label{lem:delta-small}
    For any $ \delta > 0  $, if $i$ is large enough, then for any $\vv \in \mathfrak{g}_i$ with $\vert \vv \vert _{B_i}  \in [0, \delta  ]$,  one has
    \[      \psi _i (\exp (\vv)) \in \overline{B}_{ 3  \delta /2 } (0) .                          \]
\end{Lem}

\begin{proof}
    We first deal with the case $\delta \in (0, 1/2]$.   Assume by contradiction there is a sequence $\vv _i \in \mathfrak{g}_i$ with $\vert \vv _i \vert _{B_i} \in [0, \delta]$, and $\psi_i (\exp (\vv_i)) \notin \overline{B}_{3 \delta  /2 } (0) $ for infinitely many $i's$. After passing to a subsequence, we can assume this happens for all $i \in \mathbb{N}$. 

    By definition of $\vert \cdot \vert _{B_i}$, we know that $g_i : = \exp (\vv _i ) \in B_i$ for all $i$, so after passing to a subsequence, we have  $ \psi_i (g_i) \to  x $    for some $x \in \overline{B}_1(0) \backslash  B_{3 \delta /2}(0) $.   Let 
    \[ m_0 : = \left\lfloor  \frac{1}{  \delta} \right\rfloor > \frac{2}{3 \delta } .\] 
    Then $g_i ^{m_0} \in B_i$, so 
    \[ \lim_{i \to \infty } \psi_i (g_i ^{m_0}) = m_0 \cdot  x \in \overline{B}_1(0) .     \]
    Therefore  
    \[     x \in \overline{B}_{1/m_0}(0) \subset B_{3 \delta /2 } (0) ,   \] 
    which is a contradiction. 

    Now consider $\delta > 0 $ arbitrary and pick $n_0 : = \lceil 2 \delta  \rceil $. For any sequence $\vv _i \in \mathfrak{ g}_i$ with $\vert \vv _i \vert _{B_i} \in [0, \delta ]$, one has 
    \[   \vert \vv _i / n_0 \vert _{B_i} \leq \delta / n_0 \leq 1/2 ,    \]
    so by the first part one has
    \[        \psi_i (\exp (\vv _i / n_0)) \in \overline{B}_{3\delta / 2 n_0} (0)                     \]
    for $i$ large enough. Therefore by \eqref{item:ga-3}, one has
    \[    \psi_i (\exp (\vv_i)) \in \overline{B}_{3 \delta / 2 }(0)                      \]
    for $i$ large enough. 
\end{proof}

\begin{Lem}\label{lem:delta-large-pre}
    For any $ \delta \in (0, 1/ 4 ] $, if $i$ is large enough, then for any $\vv \in \mathfrak{g}_i$ with $\vert \vv \vert _{B_i}  \in [\delta , 1 / 4  ] $,  one has 
    \[      \psi _i (\exp (\vv)) \notin B_{ \delta / 2} (0) .                          \]
\end{Lem}

\begin{proof}
    Assume by contradiction there is a sequence $\vv _i \in \mathfrak{g}_i$ with $\vert \vv _i \vert _{B_i} \in [ \delta, 1 / 4 ]  $, and $\psi_i (\exp (\vv_i)) \in B_{\delta /2 } (0) $ for infinitely many $i's$.  After passing to a subsequence, we can assume this happens for all $i\in \mathbb{N}$. Moreover,  we can assume  $   \psi_i (\exp(\vv_i)) \to  x $  for some $x \in \overline{B}_{  \delta / 2 }(0) $.   Set  
    \[  \tau_{B_i} (\vv _i )   =  m_i   +  s_i ,\] 
    with $m _i \in \mathbb{N}$ and $ s _i \in [0, 1 )$. Since $m_i \leq  1 / \delta $, after passing to a subsequence, we can assume $m_i$ actually doesn't depend on $i$.  By \eqref{item:ga-3},  one has
    \[  \lim_{i \to \infty } \left[  \psi_i (\exp (s_i \vv_i)) + \psi_i (\exp (m_i\vv_i))  - \psi_i  (\exp (\tau_{B_i} (\vv_i) \vv_i)) \right]  = 0  .                \]
    By Lemma \ref{lem:delta-small}, the first summand lies in $B_{ 3/8 }(0)$ for $i$ large enough, and by Lemma \ref{lem:convergent-sequences} the second summand converges to $m_i \cdot x \in \overline{B}_{1/2}(0)$. Therefore the last summand lies in $\overline{B}_{9/10}(0)$ for $i$ large enough, which by \eqref{item:ga-4} means $\exp (\tau_{B_i} (\vv_i)\vv_i) \in B_i$, contradicting \eqref{eq:tau-notin}.
\end{proof}

\begin{Lem}\label{lem:delta-large}
    For any $ \delta > 0  $, if $i$ is large enough, then for any $\vv \in \mathfrak{g}_i$ with $\vert \vv \vert _{B_i}  \in [\delta , 1 / \delta ] $,  one has 
    \[      \psi _i (\exp (\vv)) \notin B_{  \delta / 4 } (0) .                          \]
\end{Lem}

\begin{proof}
If the result fails, after passing to a subsequence one can find  $\vv _i \in \mathfrak{g}_i$ with $\vert \vv _i \vert _{B_i} \in [\delta , 1 /\delta]$ but 
\begin{equation}\label{eq:psi-shrink}
    \psi _i (\exp (\vv _i ) ) \in B_{ \delta / 4 } (0)  \text{ for all }i.
\end{equation}
Set $m_0 : =  \lceil 4 / \delta  \rceil $ and notice that $\vert \vv_i / m_0 \vert _{B_i}  \in [ \delta / m_0 , 1/4  ]$ for all $i$, so by Lemma \ref{lem:delta-large-pre} we have 
    \[   \psi _i ( \exp (\vv_i / m_0 )  ) \notin B_{\delta / 2 m_0 } (0)   \]
    for $i$ large enough. After passing to a subsequence, we can assume $\psi_i (\exp (\vv _i/ m_0 ))\to y$ for some  $y \notin  B_{\delta / 2m_0}(0)$. Hence
\[       \psi_i (\exp(\vv _i ))\to m_0y  \notin  B_{\delta / 2 } (0),        \]
which contradicts \eqref{eq:psi-shrink}.  
\end{proof}

\begin{Lem}\label{lem:quasi-norm}
    If $i$ is sufficiently large, for $\vv_1, \ldots , \vv_m \in \mathfrak{g}_i$ one has
    \begin{equation}\label{eq:almost-triangle}
 \vert \vv _1 + \ldots + \vv_m \vert _{B_i} \leq 2 C_0 \sum_{j =1 }^m   \vert \vv_j \vert _{B_i}  .         
    \end{equation}
\end{Lem}

\begin{proof}
   By  Theorem \ref{thm:gleason}, for each $t > 0 $ and  $\ell \in \mathbb{N}$ one has 
   \[      \Vert    \exp (t \vv_1 / \ell ) \cdots (t\vv_m / \ell )   \Vert _{B_i} \leq  C_0 \sum_{j=1}^m \Vert \exp (t \vv_j / \ell ) \Vert _{B_i}    ,   \]
    provided $i$ is large enough. Using the formula 
    \[    \exp (t ( \vv _1 + \ldots +  \vv_m )  )  = \lim_{\ell \to \infty} \left(   \exp (t \vv_1 / \ell) \cdots \exp (t\vv_m / \ell )    \right) ^{\ell}   ,   \]
     Lemma \ref{lem:norm-continuity}, and \eqref{eq:escape-less-than-norm}, we deduce that
    \begin{eqnarray*}
     \Vert  \exp (t ( \vv _1 + \ldots +  \vv_m )  )   \Vert_{B_i} & \leq &  2 \limsup_{\ell \to \infty} \Vert  \left(   \exp (t \vv_1 / \ell) \cdots \exp (t\vv_m / \ell )    \right) ^{\ell} \Vert_{B_i} \\
     & \leq & 2 C_0 \limsup _{\ell \to \infty} \, \ell  \sum _{j = 1} ^m \Vert \exp (t \vv _ j / \ell ) \Vert _{B_i}\\
     & \leq & 2 t  C_0  \sum_{j = 1} ^m    \vert \vv _j \vert_{B_i}  . 
     \end{eqnarray*}     
     Dividing over $t$ and taking the limit as $t \to 0$, by \eqref{eq:norm-definition}  we conclude \eqref{eq:almost-triangle}.
\end{proof}

\begin{Lem}\label{lem:norm}
There is a genuine norm $\Vert \cdot \Vert _i : \mathfrak{g}_i \to \mathbb{R}$ with
\[    \vert \vv \vert_{B_i}  \leq  \Vert \vv \Vert _i \leq 2 C_0 \vert \vv \vert _{B_i}                          \]
for all $\vv \in \mathfrak{g}_i$.     
\end{Lem}

\begin{proof}
    Set $D_i: =  \{  \vv \in \mathfrak{g}_i  \vert \vert \vv \vert_{B_i} \leq 1  \} $ and define 
    \[   \Vert \vv \Vert _i  : =  \inf \{ \lambda \geq  0 \vert   \vv \in \lambda \cdot \conv ( D_i)  \}  ,         \]
where $\conv (D_i)$ denotes the convex hull of $D_i$.    From Lemma \ref{lem:quasi-norm}, we have 
    \[ D_i \subset \conv (D_i) \subset 2 C_0 D_i ,\] 
    so the result follows. 
\end{proof}
Let $n_i : = \dime (G_i)$ and identify $\mathbb{S}^{n_i-1} $ with $\{ \vv \in \mathfrak{g}_i \, \vert \, \Vert \vv \Vert _i = 1 /2  \}$.  Define $f_i : \mathbb{S}^{n_i-1} \to \mathbb{R}^k$ as 
\[    f_i (\vv) : = \psi _i ( \exp (\vv) ).    \]
Note that $f_i (- \vv) =  - f_i (\vv)$ for all $\vv \in \mathbb{S}^{n_i-1}$ (see Remark \ref{rem:symmetric}). By Lemma \ref{lem:norm}, we have 
\[     1/ 4 C_0 \leq   \vert \vv \vert _{B_i}  \leq 1/2                 \]
for all $\vv \in \mathbb{S}^{n_i-1}$, so  Lemma \ref{lem:delta-large} implies
\[   f_i (\mathbb{S}^{n_i-1}) \cap B_{1/ 16C_0}(0) = \emptyset                 \]
for $i$ large enough.

By Corollary \ref{cor:approx-continuity}, $\psi _i : B_i  \to \mathbb{R}^k$ is $1/ 32C_0$-continuous for $i$ large enough, so $f_i :  \mathbb{S}^{n_i-1} \to \mathbb{R}^k$ is also $1/ 32C_0$-continuous for such $i$'s.  By Theorem \ref{thm:abu}, we have $k \geq n_i$ for $i$ large enough, concluding the proof of Theorem \ref{thm:model}.     
\end{proof}

\begin{proof}[Proof of Theorem \ref{thm:main}]
    By the Myers--Steenrod theorem, $G_i$ is a Lie group, and by \cite[Theorem 1.14]{colding-naber} (or \cite{guijarro-santos, sosa} for RCD spaces)  so is $G$. 
    By Theorem \ref{thm:egh-to-ga}, the maps $\phi_i : G_i \to G$ given by the definition of equivariant Gromov--Hausdorff convergence are good approximations. By \cite[Theorem 0.8]{pan-rong} (or \cite[Theorem 93]{santos-zamora} for RCD spaces), the groups $G_i$ have the NSS property. By Theorem \ref{thm:model}, we conclude 
    \[     \dime (G) \geq \limsup _{i \to \infty} \dime (G_i) .  \]
\end{proof}

\section{$\rcd$ spaces with large isometry groups}\label{sec:rcd-proofs}

In this section, we prove the results from Section \ref{sec:rcd}. We begin with an elementary observation regarding groups acting transitively on $\rcd$ spaces, which easily follows from \cite[Theorem 2.1]{fukaya}.

\begin{Lem}\label{lem:trans}
    Let $(X, d, \mm )$ be an $\rcd (K,N)$ space and let $G \leq \iso (X)$ be a closed subgroup. Then the following are equivalent.
    \begin{enumerate}
        \item $G$ acts transitively on $X$.\label{item:trans-1}
        \item There is $p \in X$ and a sequence  $\lambda_i \to \infty$ such that 
        \[     (   \lambda_i X , G , p   )  \xrightarrow{eGH} (Y, G_{\infty}, q)  ,           \]
        with $Y/G_{\infty} $ compact.  \label{item:trans-2}
    \end{enumerate}
\end{Lem}

\begin{comment}
\begin{proof}
\eqref{item:trans-1} implies \eqref{item:trans-3}, and \eqref{item:trans-3} implies \eqref{item:trans-2}, so we only need to show that \eqref{item:trans-2} implies \eqref{item:trans-1}. If $G$ does not act transitively on $X$, then the quotient space $X/G$ is not a point. By Theorem \ref{thm:fukaya}, we have
\begin{equation}\label{eq:fukaya-lemma}
     ( \lambda_i X / G , [p]  ) \xrightarrow{pGH} ( Y / G_{\infty} , [q ]) .    
\end{equation}
Since $X/G$ is a geodesic space that is not a point, the pointed Gromov--Hausdorff limit of $(\lambda_i X/ G , [p])$ cannot be a point, contradicting \eqref{eq:fukaya-lemma}. 
\end{proof}
\end{comment}

\begin{proof}[Proof of Theorem \ref{thm:max-symmetry}] 
    Let $x$ be a regular point and take $\lambda _i \to \infty$. By Theorem \ref{thm:fy}, after taking a subsequence, we have
    \[   (\lambda_i X, G, x) \xrightarrow{eGH} (\mathbb{R}^n, G_{\infty}, 0 )        \]
    for some closed subgroup $G_{\infty} \leq \iso (\mathbb{R}^n) $. By Theorem \ref{thm:main} and Remark \ref{rem:rcd}, we have 
    \begin{equation}\label{eq:max-symmetry-2}
      \dime (G) \leq \dime (G_{\infty} ) \leq \dime (\iso (\mathbb{R}^n)) = \frac{n(n+1)}{2} .    
    \end{equation}
    If equality holds throughout \eqref{eq:max-symmetry-2}, then $G_{\infty}$ contains the identity connected component of $\iso (\mathbb{R}^n)$, so it acts  transitively on $\mathbb{R}^n$.  By  Lemma \ref{lem:trans}, $G$ acts transitively on $X$, making it a homogeneous Riemannian manifold of dimension $n$. Then the rigidity of the smooth case applies. 
\end{proof}

Before we prove Theorem \ref{thm:not-transitive}, we need a basic result from Euclidean geometry.

\begin{Lem}\label{lem:euclidean-nt}
    Let $G \leq \iso (\mathbb{R}^n)$ be a closed subgroup not acting transitively. Then 
    \[   \dime (G) \leq \frac{n(n-1)}{2} ,   \]
    with equality only if $\mathbb{R}^n/G$ is $1$-dimensional.
\end{Lem}

\begin{proof}
Fix an orbit $M ^k \subset \mathbb{R}^n$  and $p \in M$.  The action $G \to \iso (M)$ has Kernel contained in 
\[ \{ g \in \iso (\mathbb{R}^n) \vert  \, g p = p , \, d_pg \vert _{TpM} = \Id_{TpM} \} \cong O (n-k). \]
Then a direct computation shows
\begin{eqnarray*}
  \dime (G) & \leq &      \dime (\iso (M) ) + \dime (O (n-k))         \\
  &\leq & \frac{k (k+1)}{2} + \frac{(n-k)(n-k-1)}{2} \\
  &\leq&  \frac{n (n-1)}{2} ,     
\end{eqnarray*}
with equality only if $\dime (M) = n-1$. In such a case,  any line perpendicular to $M$ projects surjectively onto $\mathbb{R}^n/G$.
\end{proof}

\begin{proof}[Proof of Theorem \ref{thm:not-transitive}] 
    Let $x$ be a regular point and take $\lambda _i \to \infty$. By Theorem \ref{thm:fy}, after taking a subsequence, we have
    \[   (\lambda_i X, G, x) \xrightarrow{eGH} (\mathbb{R}^n, G_{\infty}, 0 )        \]
    for some closed subgroup $G_{\infty} \leq \iso (\mathbb{R}^n) $.  By Lemma \ref{lem:trans}, $G_{\infty}$ does not act co-compactly on $\mathbb{R}^n$, so Lemma \ref{lem:euclidean-nt} implies \eqref{eq:nt-max-symmetry}, with equality only if $\mathbb{R}^n / G_{\infty}$ is isometric to a real interval. 
    We will now show that the quotient space $X/G$ is a $1-$dimensional manifold, possibly with boundary.
Let $p\colon (X,d) \rightarrow (X/G,d^\ast)$ be the quotient map and take $[x_0]\in X/G, r>0.$ Let
\[
\Gamma := \lbrace \gamma \in \Geo(X/G)\,|\, \gamma_0\in B_r([x_0]), \gamma_1 \in B_{2r}([x_0])-B_r([x_0])  \rbrace
\]
Via a selection argument we can find a lift $\Tilde{\Gamma}\subset \Geo(X).$ Notice that given a geodesic $\Tilde{\gamma}\colon [0,1]\rightarrow X$ and some $g\in G$ the curve $(g\Tilde{\gamma})_t:= g\Tilde{\gamma}_t$ is a geodesic. Let. $G(\Tilde{\Gamma})= \cup_{g\in G}g\Tilde{\Gamma},$ $x_0\in p^{-1}([x_0]),$ and notice that
\[
e_0(G(\Tilde{\Gamma}))\cap B_r(x_0) = B_r(x_0),
\]
where $e_0\colon \Geo(X)\rightarrow X$ is the evaluation map at $t=0.$
We now take the measure $\mu_0 := \frac{\mathfrak{m}\llcorner B_{r}(x_0)}{\mathfrak{m}(B_r(x_0))}$ and a sequence of totally atomic measures $\nu_n:=\sum_{i=1}^n\alpha_i \delta_{x_{i,n}}  $ satisfying:
\begin{itemize}
    \item $x_{i,n}= \Tilde{\gamma}_0$ for some $\Tilde{\gamma}\in G(\Tilde{\Gamma}),$
    \item $\supp \nu_n\subset \supp \nu_{n+1}$ for all $n,$
    \item $\mathbb{W}_2(\nu_n,\mu_0)\rightarrow 0$ as $n\rightarrow \infty.$
\end{itemize}
For each $\nu_n$ we define a measure $\eta_n$ in the following way: Given an atom $x_{i,n}$ of $\nu_n$ pick a geodesic $\Tilde{\gamma}^{1,n}\in G(\Tilde{\Gamma})$ such that $\Tilde{\gamma}^{1,n}_0=x_{i,n},$ then form the measure
\[
\eta_{n}:= \sum_{i=1}^n \alpha_i\delta_{y_{i,n}}, \quad \text{ where } y_{i,n}= \Tilde{\gamma}^{i,n}_1.
\]
Observe that the supports of the measures $\eta_n$ are contained in the closure of the ball $B_{2r}(x_0)$ so then the sequence $\lbrace \eta_n \rbrace_{n\in \mathbb{N}}$ is tight. We can then assume, by passing to a subsequence if necessary, that they converge to some measure $\mu_1.$
By construction of these measures we have that 
$\mathbb{W}_2(\nu_n,\eta_n)= \mathbb{W}_2(p_{\#}\nu_n,p_{\#}\eta_n),$ meaning that the transport is done by moving along geodesics that project to geodesics in $X/G.$ Taking the limit it follows that the transport between $\mu_0$ and $\mu_1$ is also realized by moving along these kinds of geodesics. If we take $\mu_{1/2}$ the midpoint between $\mu_0$ and $\mu_1,$ as $\mu_0\ll \mathfrak{m},$ it follows that $\mu_{1/2}\ll \mathfrak{m}$ and so $\mathfrak{m}(\supp \mu_{1/2})>0.$
Recall that $\mathcal{R}_n$ denotes the regular set of $X$ and that it is of full measure, then $\mathfrak{m}(\mathcal{R}_n\cap \supp\mu_{1/2} )> 0.$ Hence we can find a point $x\in X$ that is both  regular and such that there exists a geodesic $\Tilde{\gamma}$
satisfying that $\Tilde{\gamma}_{1/2}=x$ and that its projection, $\gamma,$ is still a geodesic in $X/G.$ Observe that the only tangent space at $[x]$ must be $\mathbb{R}.$

Let $\lambda_i\rightarrow \infty$ and $r>0.$ Take 
\[
[y_i] \in B_{r_i}([x])\cap(N_{r_i/2}(\gamma [0,1])-N_{r_i/4}(\gamma [0,1])),
\]
where $r_i = r/\lambda_i$  and $N_{s}(\gamma [0,1])= \cup_{[z]\in \gamma [0,1]}B_s([z]).$

Observe that for all $i$ we have that $d^\ast([y_i], \gamma [0,1]\cap \bar{B}_{r_i}([x]))\geq r_i/4.$ We also have that points in $\gamma [0,1]$ must converge to $\mathbb{R}.$ But then the sequence of the points $[y_i]$ must converge to some $[y_{\infty}]$ such that its distance to $\mathbb{R}$ must be strictly positive. As the only possible tangent of $[x]$ is $\mathbb{R}$ this is a contradiction.
Then for all $r>0$ we must have that 
\[
B_{r_i}([x])\cap(N_{r_i/2}(\gamma [0,1])-N_{r_i/4}(\gamma [0,1])) = \emptyset
\]
This yields that the interior of $\gamma [0,1]$ must be non-empty. Then there exists an open set in $X/G$ that is homeomorphic to $\mathbb{R},$ recalling that $X/G$ is non-branching it follows that then it must be a $1-$dimensional manifold, possibly with boundary.

    \color{black}

\end{proof}

\begin{proof}[Proof of Theorem \ref{thm:dimension-difference}] 
Let $x \in X$ be an  $n$-regular point such that $[x] \in X/G$ is $m$-regular, and let $\lambda_i \to \infty$. By Theorem \ref{thm:fy}, after taking a subsequence, we have
    \[   (\lambda_i X, G, x) \xrightarrow{eGH} (\mathbb{R}^n, G_{\infty}, 0 ) , \]
    for some closed subgroup $G_{\infty} \leq \iso (\mathbb{R}^n) $. By \cite[Proposition 44]{santos-zamora},  $\mathbb{R}^n$ admits a $G_{\infty}$-invariant splitting  $\mathbb{R}^m \times \mathbb{R}^{n-m}$, for which $G_{\infty}$ acts trivially on the first factor. Finally, by Theorem \ref{thm:main} and Remark \ref{rem:rcd}, we have 
    \[   \dime (G) \leq \dime (G_{\infty} ) \leq  \frac{(n-m)(n-m+1)}{2} .                   \]    
\end{proof}

\section*{Acknowledgements}
 The authors would like to express their gratitude to Jiayin Pan and Logan Richard for helpful discussions.  The authors would also like to thank Conrad Plaut for interesting conversations that led to Proposition \ref{pro:uniform-continuity} and Theorem \ref{thm:abu}. JNZ wishes to thank the PAPIIT-UNAM Project IA103925 for financial support. 
JSR thanks the support from grant PID2024-158664NB-C21 from Agencia Estatal de Investigaci\'on (Spain).
 
\printbibliography

\end{document}